\documentclass[12pt]{article}
\usepackage{amsmath}
\usepackage{amssymb}
\usepackage{latexsym,mathtools}
\usepackage{amsthm}
\usepackage{tikz}
\usepackage{tabularx}
\usepackage{booktabs}
\usepackage{mathrsfs}
\usepackage{enumitem}
\usetikzlibrary{calc}
\usepackage{xcolor}
\usepackage{ytableau} 
\usepackage{graphicx}
\usepackage{algorithm,algorithmic}
\usepackage[colorlinks,
linkcolor=red,
anchorcolor=blue,
citecolor=green]{hyperref}

\usepackage{epic}

\newtheorem{thm}{Theorem}[section]
\newtheorem{lem}[thm]{Lemma}
\newtheorem{prop}[thm]{Proposition}
\newtheorem{cor}[thm]{Corollary}
\newtheorem{conj}[thm]{Conjecture}
\newtheorem{rem}[thm]{Remark}

\newtheorem{prob}[thm]{Problem}

\allowdisplaybreaks

\newcommand{\des}{\mathrm{des}}

\newcommand{\PF}{\mathcal{PF}}
\newcommand{\PPF}{\mathcal{PPF}}
\newcommand{\type}{\mathrm{type}}
\newcommand{\PP}{p}
\newcommand{\PPP}{\tilde{p}}

\newcommand{\E}{\mathcal{E}}
\newcommand{\NC}{\mathrm{NC}}

\numberwithin{equation}{section}

\begin{document}

	\begin{center}
		{\large \bf Lascoux series, parking functions and noncrossing partitions}
	\end{center}
    
	\begin{center}
		Alice L.L. Gao$^1$, Xin-Bei Liu$^2$, Arthur L.B. Yang$^3$ and James J.Y. Zhao$^{4}$\\[6pt]

        $^{1}$School of Mathematics and Statistics,\\
		Northwestern Polytechnical University, Xi'an, Shaanxi 710072, P.R. China\\[6pt]
		$^{2,3}$Center for Combinatorics, LPMC\\
		Nankai University, Tianjin 300071, P. R. China\\[6pt]
		$^{4}$School of Accounting,\\
		Guangzhou College of Technology and Business,
       Foshan 528138, P.R. China\\[6pt]
		
		Email: $^{1}${\tt llgao@nwpu.edu.cn},
			   $^{2}${\tt lxb@mail.nankai.edu.cn},
               $^{3}${\tt yang@nankai.edu.cn},
               $^{4}${\tt zhao@gzgs.edu.cn},
		\end{center}
	\noindent\textbf{Abstract.}
    In the study of the generating series of Demazure characters,
    Lascoux used isobaric divided differences to define a family of polynomials $\mathcal{E}_{\sigma}(t)$ indexed by permutations $\sigma$, and asked for a satisfactory expression of these polynomials. In this paper we obtain a combinatorial interpretation of $\mathcal{E}_{\sigma}(t)$ for the permutation $\sigma=[2,3,\ldots,n,1]$ or its inverse in terms of the descent statistic of parking functions of length $n-1$. Based on this progress on Lascoux's open problem, we find that the polynomial $\mathcal{E}_{\sigma}(t)$ for this special case  coincides with 
    the $h$-polynomial $h(\Delta(\NC_W),t)$ of the order complex of the noncrossing partition lattice associated to the irreducible Coxeter group $W$ of type $A_{n-1}$. We are inspired by this coincidence to give an operator approach to $h(\Delta(\NC_W),t)$ for any finite Coxeter group $W$. As an application, we 
    completely solve an open problem on 
    $h(\Delta(\NC_W),t)$ which was proposed by  Athanasiadis, Douvropoulos and Kalampogia-Evangelinou. 
    For any $k$-divisible noncrossing partition poset $\NC^{(k)}_W$, we also obtain the interlacing symmetric decomposition property of the $h$-polynomial $h(\Delta(\NC^{(k)}_W),t)$.

\noindent \emph{AMS Classification 2020:} 05A20, 20F55, 26C10, 52B05

\noindent \emph{Keywords:} Demazure character, Lascoux series, isobaric derivative, parking function, noncrossing partition, $h$-polynomial, interlacing symmetric decomposition.

\section{Introduction}\label{sec 1}
This paper is motivated by an open problem proposed by Lascoux \cite{Lascoux-2016} in his study of Demazure characters. Let us first give an overview of related backgrounds. 
Let $\mathbb{N}$ (resp. $\mathbb{N}^+$) denote the set of nonnegative (resp. positive) integers, and let $\mathbb{C}$ denote the set of complex numbers. For any $n\in \mathbb{N}$ the flag variety $\mathcal{F}(\mathbb{C}^n)$ carries a sequence of tautological line bundles $L_1, \ldots, L_n$. Given a partition $\lambda = (\lambda_1,\ldots,\lambda_n)\in\mathbb{N}^n$, one can associate it to the line bundle
$L^\lambda:=L_1^{\otimes \lambda_1}\otimes \cdots L_n^{\otimes \lambda_n}$. It is well known that each permutation $\sigma \in \mathfrak{S}_n$ defines a Schubert subvariety $\Omega_\sigma \subset \mathcal{F}(\mathbb{C}^n)$.
A fundamental problem in algebraic geometry is to compute the dimension (or postulation) of the space of global sections of $L^\lambda$ restricted to $\Omega_\sigma$.
According to Demazure \cite{Demazure-1974}, these dimensions are given by the specialization $\mathbf{x}=\mathbf{1}$ (i.e., $x_1=\cdots=x_n=1$) of the key polynomial (also named as Demazure character) $\kappa_{\lambda, \sigma}(\mathbf{x})$, where $\mathbf{x} = (x_1, \ldots, x_n)$ is a vector of commuting variables.
Instead of focusing on a fixed line bundle $L^\lambda$, Lascoux \cite{Lascoux-2016} proposed studying all of its tensor powers $L^{k\lambda}$ simultaneously. This leads to the generating  series
\begin{align}\label{gene se}
1+\kappa_{\lambda, \sigma}(\textbf{x})t
+\kappa_{2\lambda, \sigma}(\textbf{x})t^2
+\kappa_{3\lambda, \sigma}(\textbf{x})t^3+\cdots,
\end{align}
whose specialization at $\mathbf{x}=\mathbf{1}$ encodes the dimensions of the spaces of global sections of $L^{j\lambda}$ over $\Omega_\sigma$ for all integers $j \geq 0$. Here, $j\lambda := (j\lambda_1, \ldots, j\lambda_n)$ denotes the dilation of the partition $\lambda$ by $j$. We refer to the series in \eqref{gene se} as Lascoux series. 

Given partition $\lambda$ and permutation $\sigma$, let 
\begin{align*}
\mathfrak{F}_{\lambda, \sigma}(t)=\sum_{j\geq 0} \kappa_{j\lambda, \sigma}(\textbf{1}) t^j.
\end{align*}
Based on the theory of generating functions and the fact that $\kappa_{j\lambda, \sigma}(\textbf{1})$ is a polynomial in $j$ \cite{Valentina-Kirichenko-Smirnov-2012}, we see that 
$\mathfrak{F}_{\lambda, \sigma}(t)$ is a rational function. When written in lowest terms, its numerator is denoted by  $\mathcal{E}_{\lambda, \sigma}(t)$. Lascoux noted that if  $\lambda$ is a strict partition then 
$\mathcal{E}_{\lambda, \sigma}(t)$ is of degree at most $\ell(\sigma)$, where $\ell(\sigma)$ is the length of $\sigma$.
Lascoux \cite{Lascoux-2016} proposed the following open problem.

\begin{prob}[{\cite[P. 262]{Lascoux-2016}}] \label{prob-Lascoux}
Give a satisfactory expression of ${\mathcal{E}}_{\lambda,\sigma}(t)$.
\end{prob}

This paper serves as our first attempt to study Lascoux's open problem. We would like to point out that these polynomials appear to be very interesting, even for some special cases. 
For the staircase partition $\delta_n=(n-1,n-2,\ldots,1,0)$, we may abbreviate 
${\mathcal{E}}_{\delta_n, \sigma}(t)$ as ${\mathcal{E}}_{\sigma}(t)$.
As shown by Lascoux, if $\sigma$ is a maximal element of certain Young subgroup then $\mathcal{E}_{\sigma}(t)$ is some classical Eulerian polynomial. Lascoux also invited the mathematical community to carry out a systematic study of the polynomials ${\mathcal{E}}_{\sigma}(t)$.
In this paper we give a combinatorial interpretation of 
${\mathcal{E}}_{\sigma}(t)$ for $\sigma=[n, 1,2,\ldots,n-1]$ and its inverse $\sigma^{-1}=[2,3,\ldots,n,1]$. 

It turns out that ${\mathcal{E}}_{[n, 1,2,\ldots,n-1]}(t)$ and ${\mathcal{E}}_{[2,3,\ldots,n,1]}(t)$ are related to parking functions, one of the most ubiquitous and fascinating objects in mathematics. 
For more information on parking functions, we refer the reader to Yan's excellent survey \cite{Yan-2015}. Recall that a parking function of length $n$ is a sequence $\mathbf{a} = (a_1, \ldots, a_n)$ of positive integers such that its non-decreasing rearrangement $\mathbf{b} = (b_1, \ldots, b_n)$ satisfies $b_i \le i$ for all $1\leq i\leq n$.
Let $\PF_n$ denote the set of parking functions of length $n$ and let 
\begin{align}\label{eq-gf-parkingfunction}
\PP_n(t):
=\sum_{\mathbf{a}\in \PF_n}t^{n-1-\mathrm{des}(\mathbf{a})}, 
\end{align}
where $\mathrm{des}(\mathbf{a})$ denotes the number of descents in $\mathbf{a}$, i.e., the number of indices $1\leq i\leq n-1$ such that $a_i> a_{i+1}$. 
The first main result of this paper is  as follows.

\begin{thm}\label{main-thm-com-in}
For any $n \geq 2$ we have
\begin{align}\label{first-main-eq}
{\mathcal{E}}_{[n, 1,2,\ldots,n-1]}(t)={\mathcal{E}}_{[2,3,\ldots,n,1]}(t)=\PP_{n-1}(t).
\end{align}
\end{thm}

Our proof of Theorem \ref{main-thm-com-in} relies on an operational interpretation of the polynomials ${\mathcal{E}}_{\sigma}(t)$ in terms of isobaric divided differences, as given by Lascoux \cite{Lascoux-2016}. 
Recall that, for each $1\leq i\leq n-1$, the action of the isobaric divided difference operator 
$\pi_i$ on a function $F(\mathbf{x})$ is defined by
\begin{align}\label{eq-isobaric-operator}
 \pi_iF(\mathbf{x}):=\frac{x_iF(x_1,\ldots,x_i,x_{i+1},\ldots,x_n)-x_{i+1} F(x_1,\ldots,x_{i+1},x_i,\ldots,x_n)}{x_i-x_{i+1}}.
\end{align}
Given a permutation $\sigma\in\mathfrak{S}_n$, let $\pi_\sigma$ denote the product 
$\pi_{i_1} \pi_{i_2} \cdots \pi_{i_\ell}$ provided that $\sigma$ has a reduced decomposition $s_{i_1} s_{i_2} \cdots s_{i_\ell}$, where every $s_j$ denotes the simple transposition $(j,j+1)$.
For any $\mathbf{a} = (a_1, \ldots, a_n)\in\mathbb{N}^n$, let $\mathbf{x}^{\mathbf{a}}$ stand for the monomial $x_1^{a_1}\cdots x_n^{a_n}$. Lascoux~\cite{Lascoux-2016} 
defined the composite linear operator $D_\sigma$ by sending a function $f(t)$ in the variable $t$ to
\begin{align*}
D_\sigma(f(t)):=\pi_\sigma f(t \mathbf{x}^{\delta_n})\big|_{\mathbf{x} = \mathbf{1}}    
\end{align*}
and showed that 
\begin{align}\label{4}
D_\sigma\left(\frac{1}{1 - t}\right)
=\sum_{j\geq 0} \kappa_{j\delta_n, \sigma}(\textbf{1}) t^j=\frac{\mathcal{E}_{\sigma}(t)}{(1-t)^{\ell(\sigma) +1}}.
\end{align}

As a consequence of Theorem \ref{main-thm-com-in}, we see that the generating polynomials $\PP_{n}(t)$
enjoy an operational interpretation. On the other hand,  
these polynomials also appear as $h$-polynomials of 
order complexes of noncrossing partition lattices associated to the irreducible Coxeter group of type $A_{n-1}$. 
Given a finite Coxeter group  $W$ of rank $r_W$, let $\NC_W$ denote its noncrossing partition lattice and $\Delta(\NC_W)$ the corresponding order complex. 
Let $h(\Delta(\NC_W),t)$ denote the $h$-polynomial of $\Delta(\NC_W)$, which is of degree $r_W-1$. 
By a result due to Stanley \cite{Stanley-1997}, we know that if $W$ is the irreducible Coxeter group of type $A_{n-1}$ then
$h(\Delta(\NC_W),t)={\PP}_{n}(t).$
Inspired by the operational interpretation of ${\PP}_{n}(t)$, we further provide an operator approach to $h(\Delta(\NC_W),t)$ for general $W$. As an application, we completely solve
an open problem on $h(\Delta(\NC_W),t)$ due to  
Athanasiadis, Douvropoulos and Kalampogia-Evangelinou~{\cite[Question 20]{Athanasiadis-Douvropoulos-Kalampogia-Evangelinou-2024}}, as stated in the following theorem.

\begin{thm} \label{que:ADKE_5.5(b)}
\label{Athanasiadis-Douvropoulos-Kalampogia-Evangelinou-question}
Suppose that $W$ is a finite Coxeter group and let $r=r_W-1$. If
\begin{align*}
h(\Delta(\NC_W),t)=\sum_{j=0}^{r}h_j(\Delta(\NC_W))t^j,
\end{align*}
then 
    \begin{align}\label{que}
        \frac{h_0(\Delta(\NC_W))}{h_r(\Delta(\NC_W))}\le\frac{h_1(\Delta(\NC_W))}{h_{r-1}(\Delta(\NC_W))}\le\cdots\le \frac{h_r(\Delta(\NC_W))}{h_0(\Delta(\NC_W))}.
    \end{align}
\end{thm}

Athanasiadis, Douvropoulos and Kalampogia-Evangelinou \cite[Theorem 3 (a)]{Athanasiadis-Douvropoulos-Kalampogia-Evangelinou-2024} showed that $h(\Delta(\NC_W),t)$ has only real zeros for any finite Coxeter group $W$. This implies that $\PP_n(t)$ has only real zeros for any $n\geq 2$. We are then motivated to study the real-rootedness of the descent generating polynomials over prime parking functions, whose definition was attributed to Gessel by Stanley \cite[Exercise 5.49 (f)]{Stanley-2023}. 
Recall that a prime parking function  of length $n$ is a parking function $\mathbf{a}=(a_1,\ldots,a_n)$ such that its increasing rearrangement $\mathbf{b}=(b_1,\ldots,b_n)$ satisfies $b_1=1$ and $b_i<i$ for any $2\leq i\leq n$.
Let $\PPF_n$ denote the set of prime parking functions of length $n$ and let 
\begin{align}\label{eq-ndes-prime-pfs} 
\PPP_n(t):
=\sum_{\mathbf{a}\in \PPF_n}t^{n-1-\mathrm{des}(\mathbf{a})}.
\end{align}
Throughout this paper, we use $\PP_n(t)$ and $\PPP_n(t)$ to denote the polynomials given by \eqref{eq-gf-parkingfunction} and \eqref{eq-ndes-prime-pfs}
respectively. We have the following result. 

\begin{thm} \label{PPF-RZ}
For any $n\geq 2$ the polynomial $\PPP_n(t)$ has only real zeros.
\end{thm}

Athanasiadis, Douvropoulos and Kalampogia-Evangelinou \cite[Theorem 3 (b)]{Athanasiadis-Douvropoulos-Kalampogia-Evangelinou-2024} also proved that $h(\Delta(\NC_W),t)$ has a nonnegative real-rooted symmetric decomposition with respect to $r_W-1$ for every irreducible finite Coxeter group $W$.
Motivated by their work, we further study the interlacing symmetric decomposition of 
$h(\Delta(\NC^{(k)}_W),t)$ for any $k\geq 1$, 
where the poset $\NC^{(k)}_W$ was introduced by Armstrong \cite{Armstrong-2009}
as a unified generalization of $\NC_W$ and the poset of classical $k$-divisible noncrossing partitions.
Our approach is based on the operational interpretation of the polynomials $h(\Delta(\NC^{(k)}_W),t)$, together with Br\"and\'en, Ferroni and Jochemko's recent work \cite{Branden-2024} on the preservation of properties for symmetric decompositions under Hadamard products. 
Given a polynomial $a(t)$ of degree $s\leq d$, let $\mathcal{I}_d(a)(t)=t^d a(t^{-1})$. The polynomial $a(t)$ is said to be symmetric with center of symmetry $d/2$ if $\mathcal{I}_d(a)(t)=a(t)$; in particular, if $a(t)\equiv0$, then $\mathcal{I}_d(a)(t)=a(t)$ for every $d$. As shown by Br\"and\'en and Solus~\cite{Branden-2021}, for every polynomial $f(t)$ of degree at most $n$ there exist unique symmetric polynomials $a(t)$ and $b(t)$ such that $f(t)=a(t)+tb(t)$, $\mathcal{I}_n(a)(t)=a(t)$, and $\mathcal{I}_{n-1}(b)(t)=b(t)$. The pair $(a(t),b(t))$ is called the symmetric decomposition of $f(t)$ with respect to $n$. If both $a(t)$ and $b(t)$ have only nonnegative coefficients, then $(a(t),b(t))$ is called a nonnegative symmetric decomposition of $f(t)$. We say that a polynomial is real-rooted if all of its zeros are real numbers, or if it is identically zero; the decomposition is called real-rooted if both $a(t)$ and $b(t)$ are real-rooted. For nonzero real-rooted polynomials $a(t)$ and $b(t)$, with zeros $\alpha_n\leq\cdots\leq \alpha_1$ and $\beta_m\leq\cdots\leq \beta_1$, respectively, we write $b(t)\preceq a(t)$ if $\cdots\leq \beta_2\leq \alpha_2\leq \beta_1\leq \alpha_1$ (which forces $n=m$ or $n=m+1$). The pair $(a(t),b(t))$ is called an interlacing symmetric decomposition of $f(t)$ if $b(t)\preceq a(t)$. Following the usual convention, for every real-rooted polynomial $f(t)$, we have $0\preceq f(t)$ and $f(t)\preceq0$. 
In this paper we establish the following result. 

\begin{thm} \label{interlacing-sd}
Suppose that $W$ is a finite Coxeter group of rank $r_W$. Then $h(\Delta(\NC_W),t)$ has a nonnegative real-rooted symmetric decomposition with respect to $r_W-1$, and for every integer $k\geq 2$ the polynomial $h(\Delta(\NC^{(k)}_W),t)$ has a nonnegative interlacing symmetric decomposition with respect to $r_W$. 
\end{thm}

As will be shown in subsequent sections, Theorems \ref{que:ADKE_5.5(b)}, \ref{PPF-RZ} and \ref{interlacing-sd} can be proved within a general framework, while Theorem \ref{main-thm-com-in} can be considered as the starting point of our approach. 
The main tool we use is the isobaric derivative $D$ acting on functions $f(t)$ of $t$ by 
\begin{align*}
D(f(t)):=\frac{\text{d}}{\text{d}t}\big(tf(t)\big).    
\end{align*}
Throughout this paper $D$ denotes the above defined operator. 
By the theory of rational generating functions, every operator $D_\sigma$ introduced by Lascoux~\cite{Lascoux-2016} 
can be written as a polynomial of $D$. This inspires us to introduce a class of operators, as well as a class of associated polynomials.  
Given a positive integer $n$, a nonnegative integer $m$ and a sequence $\mathbf{v} =(v_0,v_1,\ldots, v_{m})$ of nonnegative real numbers, 
define the operator
\begin{align}\label{eq:defi-O_nk}
O^{\bf{v}}_{n,m}:=(nD - v_m) \cdots(nD - v_1)(nD-v_0).
\end{align}
By acting $O^{\bf{v}}_{n,m}$ on $\frac{1}{1-t}$, we find that there exists certain polynomial $g^{\bf{v}}_{n,m}(t)$ of degree at most $m+1$ such that
\begin{align}\label{eq:defi-fnk(x)}
O^{\bf{v}}_{n,m} \left( \frac{1}{1 - t} \right)=
\frac{g^{\bf{v}}_{n,m}(t)}{(1 - t)^{m + 2}},
\end{align}
when written in lowest terms.
It turns out that many polynomials of combinatorial significance  appear as $g^{\bf{v}}_{n,m}(t)$, such as the $r$-colored Eulerian polynomial $\mathcal{A}_{n,r}(t)$ \cite{Steingrimsson-1994}, and $(\operatorname{des}_B,\operatorname{neg})$ $q$-Eulerian polynomials $\mathcal{B}_n(t,q)$ over the hyperoctahedral group \cite{Brenti-1994}.
Recall that (see \cite[Eq. (12)]{Brenti-1994} and \cite[Theorem 17]{Steingrimsson-1994}) 
\begin{align*}
\sum_{j\ge0}(rj+1)^n t^j=\frac{\mathcal{A}_{n,r}(t)}{(1-t)^{n+1}},\qquad\sum_{j\ge0}((1+q)j+1)^n t^j=\frac{\mathcal{B}_n(t,q)}{(1-t)^{n+1}}.
\end{align*}

The remainder of this paper is organized as follows. Section \ref{sec 2} is devoted to the proof of Theorem \ref{main-thm-com-in}. 
In Section \ref{sec 3} we express $\PPP_n(t)$ and $h(\Delta(\NC^{(k)}_W),t)$ in terms of the above polynomial $g^{\bf{v}}_{n,m}(t)$.
In Section \ref{sec 4} we prove some general properties of $g^{\bf{v}}_{n,m}(t)$ and then prove Theorems \ref{que:ADKE_5.5(b)} and \ref{PPF-RZ}. 
The proof of Theorem \ref{interlacing-sd} is given in Section \ref{sec 5}. Finally, we discuss the real-rootedness of $\mathcal{E}_{\lambda,\sigma}(t)$ in Section \ref{sec 6}.

\section{Lascoux’s operators and parking functions}\label{sec 2}

The aim of this section is to prove Theorem \ref{main-thm-com-in}. By \eqref {first-main-eq} it suffices to show that 
\begin{align}\label{eq_three equalities}
\frac{{\mathcal{E}}_{[n, 1,2,\ldots,n-1]}(t)}{(1-t)^{n}}=\frac{{\mathcal{E}}_{[2,3,\ldots,n,1]}(t)}{(1-t)^{n}}=\frac{\PP_{n-1}(t)}{{(1-t)^{n}}}.
\end{align}
Before plunging into the proof, let us briefly recall some basic definitions and provide some examples to illustrate these concepts.

Let $(n)_{m}$ be the falling factorial of $n$ of order $m$ for $n\in\mathbb{N}^+$ and $m\in\mathbb{N}$. The notation $\mathbb{R}_{\ge 0}$ (resp. $\mathbb{R}_{>0}$) refers to the set of nonnegative (resp. positive) real numbers, while $\mathbb Q_{\ge0}[t]$ (resp. $\mathbb Q_{>0}[t]$) denotes the set of polynomials in $t$ with nonnegative (resp. positive) rational coefficients. The meaning of $\mathbb{Z}_{\ge 0}$ is similar. Let $[n]=\{1,2,\ldots,n\}$ and write $\mathfrak{S}_n$ for the symmetric group on $[n]$.
Each $\sigma\in \mathfrak{S}_n$ is a mapping of $[n]$ onto itself. This permutation is represented in one-line notation as the word $\sigma=[\sigma(1),\sigma(2),\ldots,\sigma(n)]$. 
The length of $\sigma$, denoted $\ell(\sigma)$, is the number of its inversions, namely, $\ell(\sigma)=\#\{(i,j)\mid \sigma(i)>\sigma(j) \mbox{ and } i<j\}$. 
For each $1\le i\le n-1$, let $s_i$ denote the simple transposition that interchanges $i$ and $i+1$ and fixes all other letters of $[n]$. The simple transpositions satisfy the following Coxeter relations:
\begin{align*}
s_i^2 &= 1 && \text{for all } i; \\
s_i s_j &= s_j s_i && \text{if } |i-j| > 1; \\
s_i s_{i+1} s_i &= s_{i+1} s_i s_{i+1} && \text{for } 1 \le i \le n-2.
\end{align*}
We use the convention that products of permutations are read from left to right. Thus, for $\sigma\in\mathfrak S_n$, the product $\sigma s_i$ interchanges the entries in positions $i$ and $i+1$, namely, $\sigma s_i=[\sigma(1),\ldots,\sigma(i+1),\sigma(i),\ldots,\sigma(n)]$.
The symmetric group $\mathfrak{S}_n$ is generated by the simple transpositions $s_1,\ldots,s_{n-1}$, and any permutation $\sigma\in\mathfrak{S}_n$ can be written as a product of simple transpositions. 
A {reduced decomposition} for $\sigma$ with $\ell(\sigma)=\ell$ is an expression as a product of simple transpositions $\sigma = s_{i_1} s_{i_2} \cdots s_{i_\ell}$.
For example, $[2,3,\ldots,n,1]=s_1s_2\cdots s_{n-1}$ and $[n,1, 2,\ldots,n-1]=s_{n-1}s_{n-2}\cdots s_1$ by our convention. 

Let $\mathbf{x} = (x_1, \ldots, x_n)$ be a finite sequence
of commuting variables. The group $\mathfrak{S}_n$ acts on a function $f(\mathbf{x})$ by permuting variables.  
For each $1\le i\le n-1$, 
let $\pi_i$ denote the isobaric divided difference
operator defined by \eqref{eq-isobaric-operator}.
It is straightforward to verify that the
operators $\pi_i$ are linear and idempotent.  
Generally, for each $\sigma \in \mathfrak{S}_n$, one may define 
\begin{align*}
\pi_\sigma := \pi_{i_1} \pi_{i_2} \cdots \pi_{i_\ell},
\end{align*}
where $s_{i_1} s_{i_2} \cdots s_{i_\ell}$ is a reduced decomposition of $\sigma$, and this definition is independent of the choice of reduced expression. Throughout this paper, $\pi_i\pi_j f$ means $\pi_i(\pi_j f)$. This differs from Lascoux's notation \cite{Lascoux-2016}, where the same operation is written as $f\pi_j\pi_i$.

A partition $\lambda$ of a nonnegative integer $n$ is a sequence $(\lambda_1,\ldots,\lambda_m)\in\mathbb{N}^k$ satisifying $\lambda_1\ge\cdots \ge\lambda_m$ and $\sum \lambda_i=n$, denoted as $\lambda\vdash n$. The number of parts of $\lambda$ (i.e., the number of nonzero $\lambda_i$) is the length of $\lambda$, denoted $\ell(\lambda)$. Write $m_i(\lambda)$ for the number of parts of $\lambda$ equal to $i$. We also use $\langle1^{m_1(\lambda)}2^{m_2(\lambda)}\cdots\rangle$ to represent the partition $\lambda$. Given $\sigma \in \mathfrak{S}_{n}$ and a partition $\lambda$ of length at most $n$, we write
$\textbf{x}^{\lambda}=x_1^{\lambda_1}\cdots x_{n}^{\lambda_n}$ and define
$$
\kappa_{\lambda,\sigma}(\mathbf{x})=\pi_\sigma \textbf{x}^{\lambda}.
$$
The polynomials $\kappa_{\lambda,\sigma}(\mathbf{x})$ were introduced by Demazure \cite{Demazure-1974}, and 
extensively studied by the combinatorial community with variant names. 
These polynomials were called essential polynomials in \cite{Lascoux-Paul-1982}, standard bases in \cite{Lascoux-Paul-1989,Lascoux-Paul-1990}, key polynomials in \cite{Reiner-Shimozono-1995}, and Demazure characters in \cite{Postnikov-Stanley-2008}. Here we adhere to the terminology due to Reiner and Shimozono  \cite{Reiner-Shimozono-1995} but use a different notation.
We use an example to illustrate the computation of $\kappa_{\lambda,\sigma}(\mathbf{x})$. 
Take $\lambda=(2,1,0)$ and $\sigma=[2,3,1]$. Since $\sigma=s_1s_2$, we have
\begin{align*}
\kappa_{\lambda,\sigma}(\mathbf x)&=\pi_1\pi_2(x_1^2x_2)=\pi_1\left(x_1^2x_2+x_1^2x_3\right)\\
&=x_1^2x_2+x_1x_2^2+x_1^2x_3+x_1x_2x_3+x_2^2x_3.
\end{align*}

We continue to prove \eqref{eq_three equalities} by considering the series expansion of each of the three rational functions.  
If no confusion arises, the radius of convergence of the power series will not be explicitly stated throughout this paper. 
Let us first determine the series expansion of 
$\frac{{\mathcal{E}}_{[n, 1,2,\ldots,n-1]}(t)}{(1-t)^{n}}$. In view of \eqref{4} and the fact that
\begin{align}\label{eq-Dm}
D^m\left(\frac{1}{1-t}\right)
=\sum_{j\geq 0}(j+1)^mt^{j},
\end{align}
it is desirable to give a polynomial expression of $D_{[n, 1,2,\ldots,n-1]}$ in terms of the operator $D$. Using the notation defined by \eqref{eq:defi-O_nk}, we have the following result. For the sake of completeness, we include a proof here.
\begin{prop}\label{thm_w}
    For $n \geq 2$ we have
\begin{align}\label{eq-operator-identity}
    D_{[n, 1,2,\ldots,n-1]}=\frac{1}{n!}O^{{(0,1,\ldots,n-2)}}_{n,n-2},
\end{align}
or equivalently
\begin{align*}
\frac{{\mathcal{E}}_{[n, 1,2,\ldots,n-1]}(t)}{(1-t)^{n}}=\sum_{j\geq 0} \frac{1}{nj+1}\binom{nj+n}{n}t^j. 
\end{align*}
\end{prop}
\begin{proof}
To prove \eqref{eq-operator-identity}, it suffices to show the coincidence of the actions of three operators on $t^j$ for any nonnegative integer $j$. 
From \eqref{eq-Dm} it follows that, for any real polynomial $f(D)$ in $D$,  
\begin{align}\label{eq-D-expression}
f(D)\left(\frac{1}{1-t}\right)
=\sum_{j\geq 0}f(j+1) t^{j},
\end{align}
or equivalently, $f(D)(t^j)=f(j+1)t^j$ for any $j\geq 0$. 
Thus
\begin{align*}
\frac{1}{n!}O^{{(0,1,\ldots,n-2)}}_{n,n-2}(t^j)
=\frac{1}{n!}(n(j+1))\cdots(n(j+1)-n+2)t^j=\frac{1}{nj+1}\binom{nj+n}{n}t^j.
\end{align*}
We proceed to  consider the action of $D_{[n, 1,2,\ldots,n-1]}$ on $t^j$. Note that $[n, 1,2,\ldots,n-1]$ has a reduced decomposition $s_{n-1}\cdots s_2s_1$.
By definition, 
\begin{align*}
D_{[n, 1,2,\ldots,n-1]}(t^j)=\pi_{n-1}\cdots \pi_{2}\pi_{1}\left(x_1^{(n-1)j}x_2^{(n-2)j}\cdots x_{n-2}^{2j}x_{n-1}^jt^j\right)\Big|_{\mathbf{x}=\mathbf{1}}.
\end{align*}
It is well known that  
\begin{align*}
\pi_i (x_i^mx_{i+1}^{m'})=\sum_{l=0}^{m-m'} x_i^{m-l}x_{i+1}^{m'+l}
\end{align*}
holds for $m\geq m'\geq 0$.
Based on this identity and the fact that $\pi_i$ commutes with $x_l$ for any $l\not\in\{i,i+1\}$, 
one can check that  
\begin{align*}
&\pi_{n-1}\cdots \pi_{2}\pi_{1}\bigl(x_1^{(n-1)j}x_2^{(n-2)j}\cdots x_{n-2}^{2j}x_{n-1}^j\bigr)\bigg|_{\mathbf{x}=\mathbf{1}}
\\[5pt]
=&\pi_{n-1}\cdots\pi_{2}\biggl(\pi_{1}\bigl(x_1^{(n-1)j}x_2^{(n-2)j}\bigr)x_3^{(n-3)j}\cdots x_{n-2}^{2j}x_{n-1}^{j}\biggr)\bigg|_{\mathbf{x}=\mathbf{1}}
\\[5pt]
=&\pi_{n-1}\cdots\pi_{2}\biggl(\biggl(\sum_{j_1=0}^{j}x_1^{(n-1)j-j_1}x_2^{(n-2)j+j_1}\biggr)x_3^{(n-3)j}\cdots x_{n-2}^{2j}x_{n-1}^{j}\biggr)\bigg|_{\mathbf{x}=\mathbf{1}}
\\[5pt]
=&\sum_{j_1=0}^{j}\pi_{n-1}\cdots\pi_{2}\biggl(x_2^{(n-2)j+j_1}x_3^{(n-3)j}\cdots x_{n-2}^{2j}x_{n-1}^{j}\biggr)\bigg|_{\mathbf{x}=\mathbf{1}},
\end{align*}
Continuing in this manner, we find that 
\begin{align*}
&\pi_{n-1}\cdots \pi_{2}\pi_{1}\bigl(x_1^{(n-1)j}x_2^{(n-2)j}\cdots x_{n-2}^{2j}x_{n-1}^j\bigr)\bigg|_{\mathbf{x}=\mathbf{1}}
\\[5pt]
=&\sum_{j_1=0}^{j}\pi_{n-1}\cdots \pi_{3}\biggl(\pi_{2}\bigl(x_2^{(n-2)j+j_1}x_3^{(n-3)j}\bigr)x_4^{(n-4)j}\cdots x_{n-2}^{2j}x_{n-1}^{j}\biggr)\bigg|_{\mathbf{x}=\mathbf{1}}
\\[5pt]
=&\sum_{j_1=0}^{j}\pi_{n-1}\cdots \pi_{3}\biggl(\biggl(\sum_{j_2=0}^{j+j_1}x_2^{(n-2)j+j_1-j_2}x_3^{(n-3)j+j_2}\biggr)x_4^{(n-4)j}\cdots x_{n-2}^{2j}x_{n-1}^{j}\biggr)\bigg|_{\mathbf{x}=\mathbf{1}}
\\[5pt]
=&\sum_{j_1=0}^{j}\pi_{n-1}\cdots \pi_{3}\sum_{j_2=0}^{j+j_1}\biggl(x_3^{(n-3)j+j_2}x_4^{(n-4)j}\cdots x_{n-2}^{2j}x_{n-1}^{j}\biggr)\bigg|_{\mathbf{x}=\mathbf{1}}
\\[5pt]
=&\sum_{j_1=0}^{j}\sum_{j_2=0}^{j+j_1}\sum_{j_3=0}^{j+j_2}\cdots\sum_{j_{n-1}=0}^{j+j_{n-2}}1.
\end{align*}
Note that the last term counts the number of nonnegative sequences $(j_0,j_1,j_2,\ldots,j_{n-1})$ satisfying $j_0=0$ and $0\le j_{i+1}\le j_{i}+j$ for any $0\le i\le n-2$. By a result as shown in \cite[P. 5]{Heubach-Li-Mansour-2014}, this number is just $\frac{1}{nj+1}\binom{nj+n}{n}$, and hence
\begin{align*}
D_{[n, 1,2,\ldots,n-1]}(t^j)=\frac{1}{nj+1}\binom{nj+n}{n}t^j.
\end{align*}
This completes the proof.
\end{proof}

The above proof  carries over mutatis mutandis to determining the following series expansion of $\frac{{\mathcal{E}}_{[2,3,\ldots,n,1]}(t)}{(1-t)^{n}}$.
Since our convention for the action of permutations and operators differs from that of Lascoux up to taking inverses, the following result is essentially equivalent to Lascoux's result on $D_{[n,1,\ldots,n-1]}$ \cite{Lascoux-2016}.

\begin{prop}[{\cite[P. 263]{Lascoux-2016}}]\label{thm_w'-1}
For integer $n \geq 2$ we have
\begin{align*}
D_{[2,3,\ldots,n,1]}=\frac{1}{n!}O^{{(0,1,\ldots,n-2)}}_{n,n-2},
\end{align*}
or equivalently
\begin{align*}
\frac{{\mathcal{E}}_{[2,3,\ldots,n,1]}(t)}{(1-t)^{n}}=\sum_{j\geq 0} \frac{1}{nj+1}\binom{nj+n}{n}t^j. 
\end{align*}
\end{prop}

\begin{rem}\label{rem-PS-operator}
Proposition \ref{thm_w'-1} can also be derived from a result due to Pitman and Stanley \cite{Stanley-Pitman-2002} on the Ehrhart polynomial of the polytope $\Pi_m(\mathbf{x})$, which is defined for arbitrary $\mathbf{x}=(x_1,\ldots,x_m)\in\mathbb{Z}_{\geq 0}^m$ by
$$\Pi_m(\mathbf{x})=\{(y_1,\ldots,y_m)\in\mathbb{R}_{\geq 0}^m \colon y_1+\cdots+y_i\leq x_1+\cdots+x_i \mbox{ for all }1\leq i\leq m\}.$$
Pitman and Stanley \cite[Theorem~13]{Stanley-Pitman-2002} showed that, for $\mathbf{x}=(a,b,\ldots,b)\in \mathbb{N}^m$,   
the Ehrhart polynomial $i(\Pi_m(\mathbf{x}),\,j)$ is given by
\begin{align}\label{eq-abpol}
i(\Pi_m(\mathbf{x}),\,j)=\frac{1}{m!}(ja+1)(j(a+mb)+m)_{{m-1}}.
\end{align}
Postnikov and Stanley \cite[P.~172]{Postnikov-Stanley-2008} pointed out that, for $\lambda=(\lambda_1,\ldots,\lambda_n)$ and $\sigma=s_1s_2\cdots s_{n-1}$, the generalized Gelfand--Tsetlin polytope $\mathcal P_{\lambda,\sigma}$ is given by 
\begin{align*}
\mathcal P_{\lambda,\sigma}=\{(t_1,\ldots,t_{n-1})\in\mathbb R^{n-1}:\lambda_i\ge t_i\ \text{ for $1\leq i\leq n-1$};\ t_1\ge t_2\ge\cdots\ge t_{n-1}\ge \lambda_n\}.
\end{align*}
Let $\mathbf d_\lambda=(\lambda_{n-1}-\lambda_n,\lambda_{n-2}-\lambda_{n-1},\ldots,\lambda_1-\lambda_2)$.
Following their remark \cite[P.~135]{Postnikov-Stanley-2008}, it is routine to show that the polytope $\mathcal P_{\lambda,\sigma}$ is exactly the Pitman--Stanley polytope $\Pi_{n-1}(\mathbf d_\lambda)$ 
by setting $y_i=t_{n-i}-t_{n+1-i}$ for $1\le i\le n-1$, with the convention $t_n=\lambda_{n}$.
It is also known that 
$\kappa_{j\lambda,\sigma}(\mathbf 1)$ is the Ehrhart polynomial $i(\Pi_{n-1}(\mathbf d_\lambda),\,j)$; see \cite[Corollary~15.2]{Postnikov-Stanley-2008}. Applying \eqref{eq-abpol} to the Pitman--Stanley polytope $\Pi_{n-1}(\mathbf d_{\delta_n})$ yields
\begin{align*}
\kappa_{j\delta_n,s_1s_2\cdots s_{n-1}}(\mathbf1)=\frac{1}{(n-1)!}(j+1)(nj+n-1)_{n-2}=\frac{1}{nj+1}\binom{nj+n}{n},
\end{align*}
as desired.
\end{rem}

We proceed to determine the series expansion of $\frac{\PP_{n-1}(t)}{(1-t)^{n}}$, as stated below. 
Although the polynomials $\PP_{n-1}(t)$ have appeared in \cite{Athanasiadis-Douvropoulos-Kalampogia-Evangelinou-2024,Stanley-1997}, such an explicit series expansion seems missing in the literature. Connections of this expansion to $h(\Delta(\NC_W),t)$ will be discussed in the next section.  

\begin{prop}\label{thm-gf-parking}
For any $n\geq 2$ we have
\begin{align*}
\frac{\PP_{n-1}(t)}{{(1-t)^{n}}}=\sum_{j\geq 0}\frac{1}{nj+1}\binom{nj+n}{n}t^j.
\end{align*}
\end{prop}
\begin{proof} 
Given a multiset $M = \{1^{m_1},\dots,n^{m_n}\}$, let $m=m_1+\cdots+m_n$, let $\mathfrak{S}_M$ denote the set of permutations of $M$ and let
\begin{align}\label{eq-gf-multiset}
\mathcal{A}_M(t):=\sum_{\sigma\in \mathfrak{S}_M}t^{m-1-\des(\sigma)}.
\end{align}
It is well known that 
\begin{align}\label{MacMahon-generating}
    \frac{\mathcal{A}_M(t)}{(1-t)^{m+1}}=\sum_{j\geq 0}\prod_{i=1}^{n}\frac{1}{m_i!}\left(j+1\right)_{m_i}t^{j}.
\end{align}

Let $\PF_n^\downarrow$ denote the set of non-increasing parking functions of length $n$. For each $\mathbf{b}\in \PF_n^\downarrow$, the type of $\mathbf{b}$, denoted by $\mathrm{type}(\mathbf{b})$, is defined to be the non-increasing rearrangement of  $(m_1(\mathbf{b}),m_2(\mathbf{b}),\ldots,m_n(\mathbf{b}))$, where $m_i(\mathbf{b})$ denotes the number of parts of $\mathbf{b}$ equal to $i$.  
As shown by Stanley \cite{Stanley-1997}, for a partition $\lambda\vdash n$ with $m_i(\lambda)$ $i$'s for all $i$, we have 
\begin{align}\label{eq-pf-type}
\#\left \{ \mathbf{b} \in \PF_n^{\downarrow} ~\big|~\type(\mathbf{b})=\lambda \right \}
=
\frac{1}{n+1}\binom{n+1}{\ell(\lambda)}\binom{\ell(\lambda)}{m_1(\lambda),m_2(\lambda),\ldots,m_n(\lambda)}.
\end{align}
Based on \eqref{eq-gf-parkingfunction}, \eqref{MacMahon-generating} and 
\eqref{eq-pf-type}, 
we obtain 
\begin{align*}
\frac{\PP_{n-1}(t)}{{(1-t)^{n}}}
=&\sum_{\mathbf{b}\in \PF_{n-1}^\downarrow}\sum_{j\geq 0}\prod_{i=1}^{n-1}\frac{\left(j+1\right)_{m_i(\mathbf{b})}}{m_i(\mathbf{b})!}t^{j}
=\sum_{\lambda\vdash n-1}\sum_{\genfrac{}{}{0pt}{2}{\mathbf{b}\in \PF_{n-1}^\downarrow}{\mathrm{type}(\mathbf{b})=\lambda}}\sum_{j\geq 0} \prod_{i=1}^{n-1} \frac{\left(j+1\right)_{\lambda_i}}{\lambda_i!}t^{j}\\%
&=\sum_{\lambda\vdash n-1}
\frac{1}{n}\binom{n}{\ell(\lambda)}\binom{\ell(\lambda)}{m_1(\lambda),m_2(\lambda),\ldots,m_{n-1}(\lambda)}
\sum_{j \geq 0}\prod_{i=1}^{n-1} \frac{\left(j+1\right)_{\lambda_i}}{\lambda_i!}t^{j}\\[5pt]
&=\sum_{j \geq 0}t^{j}\left(\sum_{\lambda\vdash n-1}
\frac{1}{n}\binom{n}{\ell(\lambda)}\binom{\ell(\lambda)}{m_1(\lambda),m_2(\lambda),\ldots,m_{n-1}(\lambda)}
\prod_{i=1}^{n-1} \frac{\left(j+1\right)_{\lambda_i}}{\lambda_i!}\right).
\end{align*}
For a fixed partition $\lambda\vdash n-1$, the factor
$$\binom{n}{\ell(\lambda)}\binom{\ell(\lambda)}{m_1(\lambda),m_2(\lambda),\ldots,m_{n-1}(\lambda)}$$
counts the number of weak compositions $(\alpha_1,\ldots,\alpha_n)$ of $n-1$ of length $n$ whose non-increasing rearrangement is $\lambda$. As usual, we write $(\alpha_1,\ldots,\alpha_n)\vDash m$ to mean that $(\alpha_1,\ldots,\alpha_n)$ is a weak composition of $m$. Thus, we have 
\begin{align*}
\frac{\PP_{n-1}(t)}{{(1-t)^{n}}}
&=\sum_{j \geq 0}t^{j}\left(\sum_{(\alpha_1,\ldots,\alpha_n)\vDash n-1}
\frac{1}{n!}\binom{n-1}{\alpha_1,\ldots,\alpha_n}\prod_{i=1}^{n}{\left(j+1\right)_{\alpha_i}}\right)\\[5pt]
&=\sum_{j\geq 0}\frac{1}{n!}(n(j+1))_{n-1} t^j=\sum_{j\geq 0}\frac{1}{nj+1}\binom{nj+n}{n}t^j,
\end{align*}
where the second equality follows from  the multinomial theorem for falling factorial powers (see \cite[Fact 1]{Creed-Cryan-2010} for instance)
\begin{align*}
(x_1+\cdots+x_l)_m=\sum_{(\alpha_1,\,\ldots,\,\alpha_l)\vDash m}\binom{m}{\alpha_1,\ldots,\alpha_l}\prod_{i=1}^l (x_i)_{\alpha_i}.
\end{align*}
This completes the proof. 
\end{proof}

We are now able to give a proof of Theorem \ref{main-thm-com-in}.

\begin{proof}[Proof of Theorem \ref{main-thm-com-in}]
The assertion follows from \eqref{eq_three equalities}, which clearly holds by Propositions \ref{thm_w}, \ref{thm_w'-1} and \ref{thm-gf-parking}.
\end{proof}

Finally, we would like to point out that Propositions \ref{thm_w} and \ref{thm_w'-1} actually confirm some special cases of Alexandersson and Alhajjar's conjecture
\cite[P.~11]{Alexandersson-Alhajjar-2019}, which asserts that $\kappa_{j\lambda,\sigma}(\mathbf{1})\in\mathbb{Q}_{\ge 0}[j]$ for each partition $\lambda$ with at
most $n$ parts and each permutation $\sigma\in\mathfrak{S}_n$. In fact, for $\sigma=[n,1,2,\ldots,n-1]$ or $\sigma=[2,3,\ldots,n,1]$, we have
\begin{align*}
\kappa_{j\delta_n,\sigma}(\mathbf{1})=\frac{1}{nj+1}\binom{nj+n}{n}\in\mathbb{Q}_{\ge 0}[j].
\end{align*}
Recently, Jochemko and Menon \cite[P.~6]{Jochemko-Menon-2026} confirmed Alexandersson and Alhajjar's conjecture for any
$312$-avoiding permutation $\sigma$ and any partition $\lambda$.

\section{Operational interpretation of polynomials} \label{sec 3}
In the preceding section we obtained an operational interpretation of $\PP_n(t)$. 
Precisely, we have 
\begin{align}\label{eq-ppn-gnmv-def}
\PP_{n}(t)=\frac{g^{{(0,1,\ldots,n-1)}}_{n+1,n-1}(t)}{(n+1)!}
\end{align}
by \eqref{eq:defi-fnk(x)}, Theorem \ref{main-thm-com-in} and Proposition \ref{thm_w}. 
The aim of this section is to show that both the non-descent generating polynomials over prime parking functions 
and the $h$-polynomials of order complexes of $k$-divisible noncrossing partition posets associated to the irreducible Coxeter groups  
also appear as the polynomial $g^{\bf{v}}_{n,m}(t)$ up to some scalar. 
This change of viewpoint turns out to be critical for proving Theorems \ref{que:ADKE_5.5(b)}, \ref{PPF-RZ} and \ref{interlacing-sd}.
At the end of this section, we also provide an operational interpretation of some Eulerian-type polynomials. 

Let us first consider the non-descent generating polynomial
$\PPP_n(t)$ over prime parking functions. We have the following result. 

\begin{prop}\label{prop-ppf-operator}
For any integer $n\ge 2$, we have
\begin{align}\label{eq-ppn-gnmv-def-prime}
\PPP_n(t)=\frac{g^{(0,1,\ldots,n-1)}_{n-1,n-1}(t)}{(n-1)\cdot n!}.
\end{align}
\end{prop}

\begin{proof} 
We give a proof along the lines of that of Proposition \ref{thm-gf-parking}.
Let $\PPF_n^\downarrow$ denote the set of non-increasing prime parking functions of length $n$.
For an integer  partition
$\lambda\vdash n$ with $m_i(\lambda)$ parts equal to $i$, Novelli and Thibon \cite[Eq. (286)]{Novelli-Thibon-2007} showed that
$$\#\left \{ \mathbf{b} \in \PPF_n^{\downarrow} ~\big|~\type( \mathbf{b})=\lambda \right \}
=\frac{1}{n-1}\binom{n-1}{\ell(\lambda)}\binom{\ell(\lambda)}{m_1(\lambda),m_2(\lambda),\ldots,m_n(\lambda)}.$$
Together with \eqref{MacMahon-generating}, this leads to 
    \begin{align*}
\frac{\PPP_n(t)}{(1-t)^{n+1}}
&=\sum_{\lambda\vdash n}
\frac{1}{n-1}\binom{n-1}{\ell(\lambda)}\binom{\ell(\lambda)}{m_1(\lambda),m_2(\lambda),\ldots,m_n(\lambda)}
\sum_{j \geq 0}\prod_{i=1}^{n-1}\frac{\left(j+1\right)_{\lambda_i}}{\lambda_i!}t^j\\[5pt]
&=\sum_{j \geq 0}t^j\left(\sum_{\lambda\vdash n}
\frac{1}{n-1}\binom{n-1}{\ell(\lambda)}\binom{\ell(\lambda)}{m_1(\lambda),m_2(\lambda),\ldots,m_n(\lambda)}
\prod_{i=1}^{n-1}\frac{\left(j+1\right)_{\lambda_i}}{\lambda_i!}\right).
\end{align*}
{As before, the factor 
$$\binom{n-1}{\ell(\lambda)}\binom{\ell(\lambda)}{m_1(\lambda),m_2(\lambda),\ldots,m_n(\lambda)}$$
counts the number of weak compositions $(\alpha_1,\ldots,\alpha_{n-1})$ of $n$ of length $n-1$ whose non-increasing rearrangement is $\lambda$.} Taking this into account and using the multinomial theorem for falling factorial powers again, we obtain
\begin{align*}
\frac{\PPP_n(t)}{(1-t)^{n+1}}
&=\sum_{j \geq 0}t^{j}\left(\sum_{(\alpha_1,\ldots,\alpha_{n-1})\vDash n}
\frac{1}{(n-1)n!}\binom{n}{\alpha_1,\ldots,\alpha_{n-1}}\prod_{i=1}^{n-1}{\left(j+1\right)_{\alpha_i}}\right)\\[5pt]
&=\frac{1}{(n-1)n!}\sum_{j \geq 0}\big((n-1)(j+1)\big)_{n}t^j.
\end{align*}
In view of \eqref{eq:defi-O_nk} and \eqref{eq-D-expression}, we get
\begin{align*}
\frac{\PPP_n(t)}{(1-t)^{n+1}}&=\frac{ O^{(0,1,\ldots,n-1)}_{n-1,n-1}}{(n-1)n!}\left(\frac{1}{1-t}\right),
\end{align*}
which gives the desired result by \eqref{eq:defi-fnk(x)}. 
\end{proof}

Next we interpret the $h$-polynomial $h(\Delta(\NC^{(k)}_W),t)$ associated to any finite irreducible Coxeter group $W$ 
as the polynomial $g^{\bf{v}}_{n,m}(t)$.  
Before proceeding to the proof, let us recall some related definitions and results. 
Given a finite poset $P$, let $\Delta(P)$ denote the order complex of $P$, which is an abstract simplicial complex by taking its elements as vertices and its chains as faces.
Suppose that the length of the longest chain in $P$ is $n-1$.
The $h$-polynomial of $\Delta(P)$ is given by
\begin{align*}
h(\Delta(P),t) := \sum_{i=0}^{n} f_{i-1}(\Delta(P))\, t^i (1-t)^{n-i},
\end{align*}
where $f_{i-1}(\Delta(P))$ counts the number of $(i-1)$-dimensional faces of $\Delta(P)$. Following \cite{Athanasiadis-Douvropoulos-Kalampogia-Evangelinou-2024} we let $\mathcal{Z}(P,t)$ denote the zeta polynomial of $P$ defined by letting  $\mathcal{Z}(P,0)=1$ and for $\ell\geq 1$ letting $\mathcal{Z}(P,\ell)$ count the number of multichains of length $\ell-1$. 
It is known that
\begin{align}\label{eq-h-link-zeta}
    \sum_{j\ge 0}\mathcal{Z}(P,j)t^j=\frac{h(\Delta(P),t)}{(1-t)^n}.
\end{align}
Assume that $W$ is a finite irreducible Coxeter group of rank $r_W=n$. 
Armstrong \cite[Theorem 3.6.9]{Armstrong-2009} proved that the zeta polynomial of $NC^{(k)}_W$ is given by 
\begin{align}\label{eq-irr-zeta-pol}
\mathcal{Z}(\NC^{(k)}_W,m)=\frac{1}{|W|}\prod_{i=1}^n (kmh+d_i)=\frac{1}{|W|}\prod_{i=1}^n ((m+1)kh-(kh-d_i)),
\end{align}
where $h$ is the Coxeter number of $W$ and $d_1,d_2,\ldots,d_n$ are its degrees. 
Combining \eqref{eq:defi-fnk(x)}, \eqref{eq-D-expression}, \eqref{eq-h-link-zeta} and \eqref{eq-irr-zeta-pol}, we  obtain the following result. 

\begin{prop}\label{prop-hpol-operator}
Given an irreducible Coxeter group $W$ with Coxeter number $h$ and rank $n$, let $d_1,d_2,\ldots,d_n$ be its degrees with $0<d_1\le d_2\le \cdots \le d_{n-1}\le d_n=h$.
Then   
\begin{align*}
h\left(\Delta\left(\NC^{(k)}_W\right),t\right)=\frac{1}{|W|}g^{(kh-d_n,kh-d_{n-1},\ldots,kh-d_1)}_{kh,n-1}(t).
\end{align*}
\end{prop}

Finally, we interpret some Eulerian-type polynomials as $g^{\mathbf v}_{n,m}(t)$, including the $r$-colored Eulerian polynomials \cite{Steingrimsson-1994} and the $(\operatorname{des}_B,\operatorname{neg})$
$q$-Eulerian polynomials over the hyperoctahedral group \cite{Brenti-1994}. Let $\{f_n(t)\}_{n\ge0}$ be a sequence of real polynomials with $\deg(f_n(t))=n$ and $f_0(t)=1$ satisfying
\begin{align}\label{eq_Eulerian recurrence}
f_{n+1}(t)=((\alpha n+\beta)t+\gamma)f_n(t)+\alpha t(1-t)f_n'(t),
\end{align}
where $\alpha,\beta,\gamma$ are nonnegative integers. This recurrence has been studied in
\cite{Liu-Liu-Ma-Zhang-2025}. In particular, as noted in \cite[Sections 3.2 and 3.3]{Liu-Liu-Ma-Zhang-2025}, the $r$-colored Eulerian polynomials $\mathcal{A}_{n,r}(t)$ correspond to $(\alpha,\beta,\gamma)=(r,r-1,1)$, while the $(\operatorname{des}_B,\operatorname{neg})$ $q$-Eulerian polynomials $\mathcal{B}_n(t,q)$ correspond to $(\alpha,\beta,\gamma)=(1+q,q,1)$. We have the following result. 
\begin{prop}\label{eq-eulerianpol-operator} 
Let $(f_n(t))_{n\ge 0}$ be the polynomial sequence defined by \eqref{eq_Eulerian recurrence}. If $\alpha=\beta+\gamma$, then 
$f_n(t)=g^{(\beta,\ldots,\beta)}_{\alpha,n-1}(t)$; if $2\alpha=\beta+\gamma$, then $f_n(t)=\alpha^{-1}g^{(0,\alpha-\gamma,\ldots,\alpha-\gamma)}_{\alpha,n}(t)$. 
In particular, we have 
\begin{align}\label{eq-AB-interpret}
 \mathcal{A}_{n,r}(t)=g^{(r-1,\ldots,r-1)}_{r,n-1}(t),\qquad \mathcal{B}_n(t,q)=g^{(q,\ldots,q)}_{1+q,n-1}(t).   
\end{align}
\end{prop}
\begin{proof}
Assume that $\alpha$ divides $\beta+\gamma$, and write $\rho=(\beta+\gamma)/\alpha$. Then the recurrence 
\eqref{eq_Eulerian recurrence} can be written as
\begin{align*}
 (\alpha D-(\alpha-\gamma))\frac{f_n(t)}{(1-t)^{n+\rho}}=\frac{f_{n+1}(t)}{(1-t)^{n+\rho+1}},  
\end{align*}
from which one can directly verify the desired result for $\rho=1$ or $\rho=2$. 
\end{proof}

\section{Real zeros and coefficient inequalities}\label{sec 4}

Given $n\geq 1, m\geq 0$ and $\mathbf{v} = (v_0, v_1, \ldots, v_m)\in \mathbb{R}_{\geq 0}^{m+1}$, let $g^{\mathbf{v}}_{n,m}(t)$ be the polynomial defined by 
\eqref{eq:defi-fnk(x)}. 
Suppose that 
\begin{align}\label{eq-g-expansion}
g^{\mathbf{v}}_{n,m}(t) = \sum_{j=0}^{m+1} c^{\mathbf{v}}_{n,m,j} t^j.
\end{align}
The aim of this section is to explore the analytic property of $g^{\mathbf{v}}_{n,m}(t)$ and to prove Theorems \ref{que:ADKE_5.5(b)} and \ref{PPF-RZ}.
To this end, let us first recall a result on the preservation of real-rootedness under Hadamard
products of the formal power series. For two formal power series $\mathfrak{A}(t)=\sum_{j\geq 0}a_jt^j$ and $\mathfrak{B}(t)=\sum_{j\geq 0}b_jt^j$ with real coefficients, their Hadamard product is defined by $(\mathfrak{A}\star \mathfrak{B})(t)=\sum_{j\geq 0}a_jb_jt^j$.
For any polynomial $q(t)\in \mathbb{R}[t]$, there exists a real polynomial $\mathcal{W}(q)(t)$ of degree at most $\deg(q(t))$, which is defined as the numerator of the rational generating function
\begin{align*}
	\sum_{j\ge 0}q(j)t^j=\frac{\mathcal{W}(q)(t)}{(1-t)^{\deg(q(t))+1}}.
\end{align*}
For $q_1(t),q_2(t)\in\mathbb{R}[t]$, the polynomial $\mathcal W(q_1q_2)(t)$ is the numerator
of the Hadamard product of $\sum_{j\geq0}q_1(j)t^j$ and $\sum_{j\geq0}q_2(j)t^j$.
The following basic result was obtained by Wagner \cite{Wagner-1992}.

\begin{thm}[{\cite[Theorem 0.2]{Wagner-1992}}]\label{wagner-thm}
    If $\mathcal{W}(q_1)(t)$ and $\mathcal{W}(q_2)(t)$ have only nonpositive real zeros, so does $\mathcal{W}(q_1q_2)(t)$.
\end{thm}
Based on Wagner's result, we are able to give a sufficient condition for the real-rootedness of $g^{\mathbf{v}}_{n,m}(t)$.

\begin{thm}\label{thm-main-realzeros} 
If $0\le v_i\leq n$ for each $0\leq i\leq m$, then $g^{\mathbf{v}}_{n,m}(t)$ has only nonpositive real zeros. 
\end{thm}

\begin{proof}
For each $0\le i\le m$, set $q_i(t):=nt+n-v_i$.
It is easy to verify that $\mathcal{W}(q_i)(t)=v_it+n-v_i$, which is either a constant or has a nonpositive zero. 
By \eqref{eq:defi-fnk(x)} and \eqref{eq-D-expression}, we see that 
\begin{align*}
g^{\bf{v}}_{n,m}(t)=\mathcal{W}(q_0q_1\cdots q_m)(t).
\end{align*}
Iterating  Theorem \ref{wagner-thm} gives the desired result. 
\end{proof}
Since $D\left(\frac{1}{1-t}\right)=\frac{1}{(1-t)^2}$, we must have $\deg g^{\mathbf v}_{n,m}(t)\le m$ whenever $v_0=0$.
In the remaining part of this section we always assume that $v_0=0$. 
By Newton's inequality for real-rooted polynomials, we know that  
the coefficient sequence $(c^{\mathbf{v}}_{n,m,0},\ldots,c^{\mathbf{v}}_{n,m,m})$ of $g^{\mathbf{v}}_{n,m}(t)$
is log-concave, namely, for any $1\leq j\leq m-1$, 
\begin{align}\label{eq-lc-coeff-gnmv}
(c^{\mathbf{v}}_{n,m,j})^2\geq c^{\mathbf{v}}_{n,m,j-1}c^{\mathbf{v}}_{n,m,j+1}.
\end{align}
The next result shows another type of inequality satisfied by $c^{\mathbf{v}}_{n,m,j}$, which was motivated by \eqref{que}.
\begin{thm}\label{thm-main-inequality} 
If $v_0=0$ and $0\le v_i<n$ for each $1\leq i\leq m$, then 
    \begin{align}\label{conj-1-inq-A}
        \frac{c^{\mathbf{v}}_{n,m,0}}{c^{\mathbf{v}}_{n,m,m}}\le\frac{c^{\mathbf{v}}_{n,m,1}}{c^{\mathbf{v}}_{n,m,m-1}}\le\cdots\le\frac{c^{\mathbf{v}}_{n,m,m}}{c^{\mathbf{v}}_{n,m,0}}.
    \end{align}
\end{thm}

Before proving the above theorem, let us first note a recurrence relation satisfied by the coefficients of $g^{\mathbf{v}}_{n,m}(t)$.
\begin{lem}\label{lem-recursion}
For given $n,m$ and $\mathbf v=(v_0,v_1,\ldots,v_m)$, let $c^{\mathbf{v}}_{n,m,k}$ be as defined by \eqref{eq-g-expansion}. 
If $v_0=0$, then for any $0\le j \le m$ we have 
\begin{align}\label{eq:rec-cki}
 c^{\mathbf{v}}_{n,m,j}=(n(m+1-j)+v_m)c^{\overline{\mathbf{v}}}_{n,m-1,j-1}+(n(j+1)-v_m)c^{\overline{\mathbf{v}}}_{n,m-1,j},
\end{align}
where $\overline{\mathbf{v}}=(v_0,v_1,\ldots,v_{m-1})$ and $c^{(v_0)}_{n,0,0}=n$.
\end{lem}

\begin{proof}
Since $g^{(v_0)}_{n,0}(t)=n$, the initial condition $c_{n,0,0}=n$ naturally holds.
By using \eqref{eq:defi-O_nk} and \eqref{eq:defi-fnk(x)}, one can verify that 
\begin{align*}
\frac{g^{\mathbf{v}}_{n,m}(t)}{(1-t)^{m+2}}&=\big(nD-v_{m}\big)\big(nD-v_{m-1}\big)\cdots \big(nD-v_{0}\big)\left(\frac{1}{1-t}\right)\\[5pt]
&=\big(nD-v_{m}\big)\left(\frac{g^{\overline{\mathbf{v}}}_{n,m-1}(t)}{(1-t)^{m+1}}\right)\\[5pt]
&=n\left(\frac{tg^{\overline{\mathbf{v}}}_{n,m-1}(t)}{(1-t)^{m+1}}\right)'-v_{m}\frac{g^{\overline{\mathbf{v}}}_{n,m-1}(t)}{(1-t)^{m+1}}\\[5pt]
=&\ n\left(\frac{g^{\overline{\mathbf{v}}}_{n,m-1}(t)+t (g^{\overline{\mathbf{v}}}_{n,m-1}(t))'}{(1-t)^{m+1}}+\frac{(m+1)t g^{\overline{\mathbf{v}}}_{n,m-1}(t)}{(1-t)^{m+2}}\right)
   -v_{m}\frac{g^{\overline{\mathbf{v}}}_{n,m-1}(t)}{(1-t)^{m+1}}\\[5pt]
=&\ \frac{\big((nm+v_{m})t+n-v_{m}\big)g^{\overline{\mathbf{v}}}_{n,m-1}(t)+nt(1-t) (g^{\overline{\mathbf{v}}}_{n,m-1}(t))'}{(1-t)^{m+2}}.
\end{align*}
Multiplying both sides by $(1-t)^{m+2}$ and then comparing the coefficients of $t^j$ yield the desired result.
\end{proof}
We proceed to prove Theorem \ref{thm-main-inequality}. 

\begin{proof}[Proof of Theorem \ref{thm-main-inequality}.]
Use induction on $m$. Note that the $m=0$ case is vacuous. 
From \eqref{eq:defi-O_nk} and \eqref{eq:defi-fnk(x)} it follows that 
$g^{{(v_0,v_1)}}_{n,1}(t)=n(n+v_1)t+n(n-v_1)$. Since $0\leq v_1<n$, 
the assertion is clear for $m=1$.
Now let $m\geq 2$ and assume the assertion for $m-1$.
It is sufficient to show that for each $0\leq j\leq m-1$,
\begin{align*}
        \frac{c^{\mathbf{v}}_{n,m,j}}{c^{\mathbf{v}}_{n,m,m-j}}\le\frac{c^{\mathbf{v}}_{n,m,j+1}}{c^{\mathbf{v}}_{n,m,m-1-j}}.
    \end{align*} 
We will use the recurrence relation \eqref{eq:rec-cki}. Since $n$ and $\mathbf{v}$ will be fixed throughout the proof, the use of $c_{m,j}$ to denote $c^{\mathbf v}_{n,m,j}$ should cause no confusion. 
By Lemma \ref{lem-recursion} and the hypothesis that $0\leq v_i<n$ for each $0\leq i\leq m$,
we see that $c_{m,j}> 0$ for any $0 \le j \le m$.
Thus it remains to prove the following equivalent inequality
    \begin{align}\label{conj-1-inq-C}
      c_{m,j+1}c_{m,m-j}-  c_{m,j}c_{m,m-1-j}\ge 0.
    \end{align}
For any $0\leq i,i' \leq m-1$, let
\begin{align*}
\mathcal{D}_{m-1,i,i'}:=c_{m-1,i}c_{m-1,m-1-i'}-c_{m-1,i'}c_{m-1,m-1-i}.
 \end{align*}
The induction hypothesis
    \begin{align*}
        \frac{c_{m-1,0}}{c_{m-1,m-1}}\le\frac{c_{m-1,1}}{c_{m-1,m-2}}\le\cdots\le\frac{c_{m-1,m-1}}{c_{m-1,0}}
    \end{align*}
implies that $\mathcal{D}_{m-1,i,i'}$ is nonnegative whenever $i\geq i'$. 
By Lemma \ref{lem-recursion}, for any $0\leq j \leq m-1$ we have 
\begin{align*}
c_{m,j}&=\big(n(m+1-j)+v_m\big)c_{m-1,j-1}+\big(n(j+1)-v_m\big)c_{m-1,j},\\
c_{m,m-j}&=\big(n(j+1)+v_m\big)c_{m-1,m-1-j}+\big(n(m+1-j)-v_m\big)c_{m-1,m-j},\\
c_{m,j+1}&=\big(n(m-j)+v_m\big)c_{m-1,j}+\big(n(j+2)-v_m\big)c_{m-1,j+1},\\
c_{m,m-1-j}&=\big(n(j+2)+v_m\big)c_{m-1,m-2-j}+\big(n(m-j)+v_m\big)c_{m-1,m-1-j}.
\end{align*}
One can verify that substituting into 
the left hand side of \eqref{conj-1-inq-C} yields
\begin{align}
       c_{m,j+1}c_{m,m-j}&- c_{m,j}c_{m,m-1-j}\notag\\
       =& (nm-nj-v_m)(nm-nj+n+v_m)\mathcal{D}_{m-1,j,j-1} \notag \\ 
       &+(nj+n-v_m)(nj+2n+v_m)\mathcal{D}_{m-1,j+1,j} \notag\\ 
      &+(nj+2n-v_m)(nm-nj+n-v_m)\mathcal{D}_{m-1,j+1,j-1} \notag \\
      &+2v_mnT_{n,m,j}, \qquad \qquad \qquad ~~~\label{conj-1-inq-D}
 \end{align}
 where 
 \begin{align*}
   T_{n,m,j}&=(c_{m-1,j} c_{m-1,m-j}+c_{m-1,j+1}c_{m-1,m-1-j}-2c_{m-1,j-1}c_{m-1,m-2-j})\\
   &+(m+1)(c_{m-1,j}c_{m-1,m-1-j} -c_{m-1,j-1}c_{m-1,m-2-j}). 
 \end{align*}
In view of the nonnegativity of $\mathcal{D}_{m-1,j,j-1}$, $\mathcal{D}_{m-1,j+1,j}$ and $\mathcal{D}_{m-1,j+1,j-1}$, 
it suffices to show that 
\begin{align}\label{eq-1-sim}
c_{m-1,j} c_{m-1,m-j}+c_{m-1,j+1}c_{m-1,m-1-j}-2c_{m-1,j-1}c_{m-1,m-2-j}\ge 0
\end{align}
and
\begin{align}\label{eq-2-sim}
c_{m-1,j}c_{m-1,m-1-j} -c_{m-1,j-1}c_{m-1,m-2-j}\ge 0
\end{align}
for any $0 \leq j \leq m - 1$. We may assume that $1 \leq j \leq m - 2$ since \eqref{eq-1-sim} and \eqref{eq-2-sim} clearly hold for $j=0$ or $j=m-1$.
Note that \eqref{eq-1-sim} is equivalent to 
\begin{align}\label{eq-3-sim}
\frac{c_{m-1,j}}{c_{m-1,j-1}} \cdot \frac{c_{m-1,m-j}}{c_{m-1,m-2-j}}+\frac{c_{m-1,j+1}}{c_{m-1,j-1}}\cdot\frac{c_{m-1,m-1-j}}{c_{m-1,m-2-j}}\ge 2.
\end{align}
From the induction hypothesis, we have
\begin{align}\label{eq-3-sim-2}
\frac{c_{m-1,j}}{c_{m-1,m-1-j}}\ge \frac{c_{m-1,j-1}}{c_{m-1,m-j}} \quad \text{and} \quad 
\frac{c_{m-1,j+1}}{c_{m-1,m-2-j}}\ge\frac{c_{m-1,j-1}}{c_{m-1,m-j}},
\end{align}
which are equivalent to
$$\frac{c_{m-1,j}}{c_{m-1,j-1}}\ge \frac{c_{m-1,m-1-j}}{c_{m-1,m-j}} \quad \text{and}\quad
\frac{c_{m-1,j+1}}{c_{m-1,j-1}}\ge \frac{c_{m-1,m-2-j}}{c_{m-1,m-j}}.$$
Substituting the above two inequalities into the left-hand side of
\eqref{eq-3-sim}, we obtain
 \begin{align*}
     \frac{c_{m-1,j}}{c_{m-1,j-1}}\cdot\frac{c_{m-1,m-j}}{c_{m-1,m-2-j}}+&\frac{c_{m-1,j+1}}{c_{m-1,j-1}}\cdot\frac{c_{m-1,m-1-j}}{c_{m-1,m-2-j}}
     \\[5pt]
    & 
   \ge \frac{c_{m-1,m-1-j}}{c_{m-1,m-j}}\cdot\frac{c_{m-1,m-j}}{c_{m-1,m-2-j}}+\frac{c_{m-1,m-2-j}}{c_{m-1,m-j}}\cdot\frac{c_{m-1,m-1-j}}{c_{m-1,m-2-j}} \\[5pt]
     & =\frac{c_{m-1,m-1-j}}{c_{m-1,m-2-j}}+\frac{c_{m-1,m-1-j}}{c_{m-1,m-j}}\\[5pt]
     & \geq \frac{c_{m-1,m-j}}{c_{m-1,m-1-j}}+\frac{c_{m-1,m-1-j}}{c_{m-1,m-j}}\geq 2,
 \end{align*}
 as desired, where the second inequality follows from  \eqref{eq-lc-coeff-gnmv}. 
Let us proceed to prove \eqref{eq-2-sim}, which is equivalent to
\begin{align*}
\frac{c_{m-1,j-1}} {c_{m-1,m-1-j}}
\le
\frac{c_{m-1,j}}{c_{m-1,m-2-j}}.
\end{align*}
By using \eqref{eq-lc-coeff-gnmv} and \eqref{eq-3-sim-2}, one can show that
\begin{align*}
\frac{c_{m-1,j-1}} {c_{m-1,m-1-j}}
=\frac{c_{m-1,j-1}}{c_{m-1,m-j}} \cdot \frac{c_{m-1,m-j}}{c_{m-1,m-1-j}} 
\le
\frac{c_{m-1,j}}{c_{m-1,m-1-j}} \cdot \frac{c_{m-1,m-1-j}}{c_{m-1,m-2-j}}
=\frac{c_{m-1,j}}{c_{m-1,m-2-j}},
\end{align*}
as desired.
This completes the proof.
\end{proof}

Now we are able to prove Theorems \ref{que:ADKE_5.5(b)} and \ref{PPF-RZ}.

\begin{proof}[Proof of Theorems \ref{que:ADKE_5.5(b)} and \ref{PPF-RZ}.]
Theorem \ref{PPF-RZ} is an immediate consequence of Proposition \ref{prop-ppf-operator} and Theorem \ref{thm-main-realzeros}.
Before proving Theorem \ref{que:ADKE_5.5(b)}, let us recall a basic result on zeta polynomials.
Suppose that $P$ is a poset whose longest chain is of length $n-1$ and $Q$ is a poset whose longest chain is of length $m-1$.
It is known that   
\begin{align}\label{eq-product}
\sum_{j\ge 0}\mathcal{Z}(P,j)\mathcal{Z}(Q,j)t^j=\frac{h(\Delta(P\times Q),t)}{(1-t)^{m+n-1}}.
\end{align}
Let $W$ be a finite Coxeter group with irreducible components $W_1,\ldots,W_\ell$. Suppose that each $W_i$ is of rank $n_i$ with Coxeter number $h_i$ and degrees $0<d_{i,1}\le d_{i,2}\le \cdots \le d_{i,n_i}=h_i$.
By Proposition \ref{prop-hpol-operator} we see that 
\begin{align}\label{eq_h star-k=1}
h\left(\Delta\left(\NC_{W_i}\right),t\right)=\frac{1}{|W_i|}g^{(h_i-d_{i,n_i},\ldots,h_i-d_{i,2},h_i-d_{i,1})}_{h_i,n_i-1}(t).
\end{align}
Let $\mathfrak{h}=h_1h_2\cdots h_{\ell}$ and $v_{i,j}=(h_i-d_{i,n_i+1-j})\mathfrak{h}/h_i$ for each $1\leq i\leq \ell$ and $1\leq j\leq n_i$. It is easy to see that
$v_{i,1}=0$ and $0\leq v_{i,j}<\mathfrak{h}$ for all $1\leq j\leq n_i$. 
By using \eqref{eq-product}, \eqref{eq_h star-k=1} and the fact that $\NC_W\cong \NC_{W_1}\times \cdots\times \NC_{W_{\ell}}$ (see \cite[Section 3.4.3]{Armstrong-2009}), one can show that 
\begin{align}\label{eq_finite Coxeter}
h\left(\Delta\left(\NC_{W}\right),t\right)=\prod_{i=1}^{\ell}\frac{h_i^{n_i}}{|W_i|\mathfrak{h}^{n_i}}g^{(v_{1,1},\ldots,v_{1,n_1},\ldots,v_{\ell,1},\ldots,v_{\ell,n_{\ell}})}_{\mathfrak{h},n_1+\cdots+n_{\ell}-1}(t).   
\end{align}
The desired inequality then follows from Theorem \ref{thm-main-inequality}.  
\end{proof}

\section{Interlacing symmetric decompositions}\label{sec 5}

The aim of this section is to prove Theorem \ref{interlacing-sd}. Note that Br\"and\'en, Ferroni and Jochemko \cite{Branden-2024} proved that interlacing symmetric decompositions are preserved under Hadamard products. 

\begin{thm}[{\cite[Theorem 1.5 and Proposition 4.6]{Branden-2024}}]\label{thm:haremard-real-brenti}
If $\mathcal{W}(q_1)(t)$ and $\mathcal{W}(q_2)(t)$ have nonnegative, interlacing symmetric decompositions with respect to $\deg(q_1(t))$ and $\deg(q_2(t))$ respectively, then $\mathcal{W}(q_1q_2)(t)$ has a nonnegative, interlacing symmetric decomposition with respect to $\deg(q_1(t))+\deg(q_2(t))$.
\end{thm}

This enables us to establish the following result. 

\begin{thm}\label{eq-results}
Let $g^{\mathbf{v}}_{n,m}(t)$ be the polynomial defined by \eqref{eq:defi-fnk(x)} for some $n\geq 1, m\geq 0$ and $\mathbf{v} = (v_0, v_1, \ldots, v_m)\in \mathbb{R}_{\geq 0}^{m+1}$ with $0\le v_i\le n$ for each $0\leq i\leq m$.
If there exists a set partition $\mathbf{B}$ of 
$\{0,1,\ldots,m\}$ into singletons or pairs such that $v_i\ge \frac{n}{2}$ for each singleton $\{i\}\in \mathbf{B}$ and 
$v_{j}+v_{j'}\ge n$ for each pair $\{j,j'\}\in \mathbf{B}$, then $g^{\mathbf{v}}_{n,m}(t)$ has a nonnegative, interlacing symmetric decomposition with respect to $m+1$.
\end{thm}

\begin{proof}
As in the proof of Theorem \ref{thm-main-realzeros}, for each $0\le i\le m$ let $q_i(t):=nt+n-v_i$ and hence $\mathcal{W}(q_i)(t)=v_it+n-v_i$.
Let us first show that 
for each singleton $\{i\}$ of the set partition $\mathbf{B}$ the polynomial $\mathcal{W}(q_i)(t)$ has a nonnegative, interlacing symmetric decomposition with respect to $1$. 
It is routine to verify that the polynomial $\mathcal W(q_i)(t)$ admits the following desired symmetric decomposition
\begin{align*}
\mathcal W(q_i)(t)=a_i(t)+tb_i(t),
\end{align*}
where
\begin{align*}
a_i(t)=(n-v_i)(1+t),\qquad b_i(t)=2v_i-n.
\end{align*}
Since $v_i\ge \frac n2$, both coefficients of $a_i(t)$ and $b_i(t)$ are nonnegative, and moreover $b_i(t)$ interlaces $a_i(t)$. Here we also allow $v_i=n$ or $v_i=n/2$ by the usual convention that the identically zero polynomial interlaces or is interlaced by any other real-rooted polynomial.

We continue to show that for each pair $\{j,j'\}$ of $\mathbf{B}$ the polynomial 
$\mathcal{W}(q_{j}q_{j'})(t)$ has a nonnegative, interlacing symmetric decomposition with respect to $2$.
By Theorem \ref{thm:haremard-real-brenti} and the hypothesis $v_{j}+v_{j'}\ge n$, we may assume that $v_{j}<n/2$, and hence $v_{j'}> n/2$. 
One can check that 
\begin{align*}
\mathcal W(q_{j}q_{j'})(t)=a_{j,j'}(t)+tb_{j,j'}(t),    
\end{align*}
where 
\begin{align*}
a_{j,j'}(t)={(n-v_{j})(n-v_{j'})t^2+2(n^2-v_{j}v_{j'})t+(n-v_{j})(n-v_{j'})},    
\end{align*}
and
\begin{align*}
b_{j,j'}(t)=n(v_{j}+v_{j'}-n)(1+t).    
\end{align*}
Note that both $a_{j,j'}(t)$ and $b_{j,j'}(t)$ have nonnegative coefficients. Furthermore, by the hypothesis, if $v_{j'}<n$ then
\begin{align*}
(n^2-v_{j}v_{j'})>(n-v_{j})(n-v_{j'}).
\end{align*}
Thus, $a_{j,j'}(-1)<0$, and the symmetric quadratic polynomial $a_{j,j'}(t)$ is real-rooted, with one zero on each side of $-1$.
Since $b_{j,j'}(t)$ is a nonnegative multiple of $(1+t)$, it follows that $b_{j,j'}(t)$ interlaces $a_{j,j'}(t)$, as desired. 
{If $v_{j'}=n$, then $a_{j,j'}(t)={2(n-v_{j})t}$ and $b_{j,j'}(t)={v_{j}(1+t)}$. In this case, $b_{j,j'}(t)$ also interlaces $a_{j,j'}(t)$.} Therefore, $\mathcal W\bigl(p_{j}p_{j'}\bigr)(t)$ has a nonnegative, interlacing symmetric decomposition with respect to $2$.

Recall that $\mathcal W\!\left(\prod_{i=0}^m q_i\right)(t)=g^{\mathbf v}_{n,m}(t)$.
By properly grouping $q_i(t)$ into pairs or singletons according to the set partition $\mathbf{B}$ and then iterating Theorem~\ref{thm:haremard-real-brenti}, we find that $g^{\mathbf v}_{n,m}(t)$
has a nonnegative, interlacing symmetric decomposition with respect to $m+1$. This completes the proof.
\end{proof}

We are almost ready to prove Theorem \ref{interlacing-sd}. However, the proof given here will depend on the value of the parameter $k$ of
$\NC^{(k)}_W$. For the case of $k=1$, the following remark is useful. 

\begin{rem}\label{rem_reciprocity}
Given a polynomial $f(t)$ of degree $r\geq 1$, let $\hat{f}(t)=t^{r+1}f(1/t)$. 
Suppose that $\hat{f}(t)=a(t)+tb(t)$ is a nonnegative, interlacing symmetric decomposition
with respect to $r+1$. 
Since $\deg(f(t))=r$, we have $\hat{f}(0)=0$, and hence $a(0)=0$. 
Thus $a(t)=tc(t)$ for some polynomial $c(t)$. 
Since $\mathcal{I}_{r+1}(a)(t)=a(t)$, it follows that $\mathcal{I}_{r-1}(c)(t)=c(t)$. Moreover, $b(t)$ is symmetric with respect to $r$, and therefore $f(t)=t^{r+1}\hat{f}(1/t)=b(t)+tc(t)$.
Finally, from $b(t)\preceq a(t)=tc(t)$ and \cite[Section 3]{Wagner-1992} it follows that $c(t)\preceq b(t)$. Thus, $f(t)$ has a nonnegative, interlacing symmetric decomposition with respect to $r$.
\end{rem}

We also need the following characterization of $\NC^{(k)}_W$, which is an immediate consequence of the component-wise definition of reflection absolute order. 

\begin{prop}\label{prop:NCk-product-irreducible}
Let $W$ be a Coxeter group with irreducible components $W_1,\,\ldots,\,W_\ell.$ 
Then for any integer $k\geq 1$ {there is an isomorphism of posets} 
\begin{align*}
\NC_W^{(k)}\cong\NC_{W_1}^{(k)}\times\cdots\times \NC_{W_\ell}^{(k)}.    
\end{align*}
\end{prop}
Now we are able to prove Theorem \ref{interlacing-sd}.

\begin{proof}[Proof of Theorem \ref{interlacing-sd}]
Let us first consider the interlacing symmetric decomposition of \(h(\Delta(\NC_W),t)\).
Let $W$ be a finite Coxeter group with irreducible components $W_1,\ldots,W_\ell$. 
Suppose that each $W_i$ is of rank $n_i$ with Coxeter number $h_i$ and degrees $0<d_{i,1}\le d_{i,2}\le \cdots \le d_{i,n_i}=h_i$. Then
\(r_W=n_1+\cdots+n_\ell=r+1\). By \eqref{eq_finite Coxeter}, for $m\ge 0$, we see that
\begin{align*}
\mathcal Z(\NC_W,m)=\prod_{i=1}^{\ell}\frac{h_i^{n_i}}{|W_i|\mathfrak{h}^{n_i}}\prod_{j=1}^{n_i}((m+1)\mathfrak{h}-(h_i-d_{i,j})\mathfrak{h}/h_i).
\end{align*}
where $\mathfrak{h}=h_1h_2\cdots h_{\ell}$. Since
\begin{align*}
(-1)^{r+1}\mathcal Z(\NC_W,-m-1)=\prod_{i=1}^{\ell}\frac{h_i^{n_i}}{|W_i|\mathfrak{h}^{n_i}}\prod_{j=1}^{n_i}((m+1)\mathfrak{h}-(\mathfrak{h}-v_{ij})),
\end{align*}
we have 
\begin{align*}
\mathcal{W}({(-1)^{r+1}\mathcal Z(\NC_W,-m-1)})(t)=\prod_{i=1}^{\ell}\frac{h_i^{n_i}}{|W_i|\mathfrak{h}^{n_i}}g^{(u_{1,1},\ldots,u_{1,n_1},\ldots,u_{\ell,1},\ldots,u_{\ell,n_{\ell}})}_{\mathfrak{h},r}(t),
\end{align*}
where $u_{i,j}=\mathfrak{h}d_{i,j}/h_i$ for each $1\leq i\leq \ell$ and $1\leq j\leq n_i$. By \cite[Figure~2.7]{Armstrong-2009}, for each fixed $i$, the degrees can be partitioned into pairs $\{d_{i,j},d_{i,j'}\}$ with $d_{i,j}+d_{i,j'}=h_i+2$, and possibly one singleton $\{d_{i,j}\}$ with $d_{i,j}=(h_i+2)/2$.
Thus, for each corresponding pair $\{u_{i,j},u_{i,j'}\}$ we have
$u_{i,j}+u_{i,j'}=\mathfrak{h}(h_i+2)/h_i\ge \mathfrak{h}$, and for the
possible singleton $\{u_{i,j}\}$ we have
$u_{i,j}=\mathfrak{h}(h_i+2)/(2h_i)\ge \mathfrak{h}/2$. Namely, the sequence 
\begin{align*}
(u_{1,1},\ldots,u_{1,n_1},\ldots,u_{\ell,1},\ldots,u_{\ell,n_{\ell}})    
\end{align*}
satisfies the hypothesis of Theorem \ref{eq-results} with $n=\mathfrak{h}$ and $m=r$. 
Therefore, the polynomial $\mathcal{W}({(-1)^{r+1}\mathcal Z(\NC_W,-m-1)})(t)$ has a nonnegative, interlacing symmetric decomposition with respect to $r+1$. By the reciprocity law of rational power series \cite[Lemma 3.15.11]{Stanley-2012}, it is straightforward
to verify that
\begin{align*}
\mathcal{W}({(-1)^{r+1}\mathcal Z(\NC_W,-m-1)})(t)=t^{r+1}h(\Delta(\NC_W),1/t).
\end{align*}
Remark \ref{rem_reciprocity} tells us that \(h(\Delta(\NC_W),t)\) has a nonnegative, interlacing symmetric decomposition with respect to $r$, as desired.

We proceed to consider the interlacing symmetric decomposition of $h(\Delta(\NC^{(k)}_W),t)$ for $k\geq 2$.
From \eqref{eq-product} and Proposition \ref{prop:NCk-product-irreducible} it follows that
\begin{align}\label{eq-chain-k-noncrossing}
h\left(\Delta(\NC_W^{(k)}),t\right)=\prod_{i=1}^{\ell}\frac{h_i^{n_i}}{|W_i|\mathfrak{h}^{n_i}}g^{(\hat{v}_{1,1},\ldots,\hat{v}_{1,n_1},\ldots,\hat{v}_{\ell,1},\ldots,\hat{v}_{\ell,n_{\ell}})}_{k\mathfrak{h},n_1+\cdots+n_{\ell}-1}(t),
\end{align}
where $\hat{v}_{i,j}=(kh_i-d_{i,n_i+1-j})\mathfrak{h}/h_i$ for each $1\leq i\leq \ell$ and $1\leq j\leq n_i$. 
Then $k\mathfrak{h}/2\le \hat{v}_{i,j}<k\mathfrak{h}$ for all $i,j$ when $k\ge2$. 
By Theorem~\ref{eq-results}, the polynomial $h(\Delta(\NC^{(k)}_W),t)$ has a nonnegative, interlacing symmetric decomposition with respect to $r_W$.
\end{proof}

By \eqref{eq-chain-k-noncrossing} and Theorem \ref{thm-main-realzeros}, we have the following result, a generalization of \cite[Theorem 3 (a)]{Athanasiadis-Douvropoulos-Kalampogia-Evangelinou-2024}.

\begin{cor}
The polynomial $h\left(\Delta(\NC_W^{(k)}),t\right)$ has only real zeros for any $k\geq 1$ and any finite Coxeter group $W$. 
\end{cor}

By \eqref{eq-AB-interpret} and Theorem \ref{eq-results}, we also obtain the following result, whose proof is omitted here. 
Note that the interlacing symmetric decomposition of $\mathcal{A}_{n,r}(t)$ is implied by \cite[Corollary 3.2 and Theorem 2.6]{Branden-2021}.

\begin{cor}
For any $n\ge 1$, $r\ge 2$ and $q\ge 1$, both $\mathcal{A}_{n,r}(t)$ and $\mathcal{B}_n(t,q)$ have nonnegative, interlacing symmetric decompositions with respect to $n$.
\end{cor}

We conclude this section by discussing the symmetry of $g^{\mathbf{v}}_{n,m}(t)$. 
For a polynomial $q(t)$ of degree $d$ with $\mathcal{I}_s\mathcal{W}(q)(t)=\mathcal{W}(q)(t)$ and $s\le d$, 
we follow \cite{Branden-2024} to define the defect $\mathrm{def}(q(t))$ of $q(t)$ as $d-s$. Br\"and\'en, Ferroni and Jochemko obtained the following result. 

\begin{thm}[{\cite[Theorem 1.3]{Branden-2024}}]\label{thm-symmetry-preservation}
If $\mathcal{W}(q_1)(t)$ and $\mathcal{W}(q_2)(t)$ are symmetric polynomials and $\mathrm{def}(q_1(t))=\mathrm{def}(q_2(t))$, then $\mathcal{W}(q_1q_2)(t)$ is symmetric, and $\mathrm{def}(q_1(t)q_2(t))=\mathrm{def}(q_1(t))=\mathrm{def}(q_2(t))$.
\end{thm}

Based on their result, we obtain a sufficient condition for the symmetry of $g^{\mathbf{v}}_{n,m}(t)$ as follows. 

\begin{thm}\label{symmetry-thm}
Let $g^{\mathbf{v}}_{n,m}(t)$ be the polynomial defined by \eqref{eq:defi-fnk(x)} for some $n\geq 1, m\geq 0$ and $\mathbf{v} = (v_0, v_1, \ldots, v_m)\in \mathbb{R}_{\geq 0}^{m+1}$ with $0\le v_i\le n$ for each $0\leq i\leq m$.
If there exists a set partition $\mathbf{B}$ of 
$\{0,1,\ldots,m\}$ into singletons or pairs such that $v_i={n}/{2}$ for each singleton $\{i\}\in \mathbf{B}$ and 
$v_{j}+v_{j'}=n$ for each pair $\{j,j'\}\in \mathbf{B}$, then $g^{\mathbf{v}}_{n,m}(t)$ is symmetric.
\end{thm}

\begin{proof}
Along the lines of the proof of Theorem \ref{eq-results} it is easy to see that 
for each singleton $\{i\}\in \mathbf{B}$ the polynomial $\mathcal{W}\left(q_i\right)(t)$ is symmetric with $\mathrm{def}(q_i)=0$, 
and for each pair $\{j,j'\}\in \mathbf{B}$ the polynomial $\mathcal{W}(q_{j}q_{j'})(t)$ is also symmetric with $\mathrm{def}(q_{j}q_{j'})=0$. 
By recalling $\mathcal W\!\left(\prod_{i=0}^m q_i\right)(t)=g^{\mathbf v}_{n,m}(t)$ and iterating Theorem \ref{thm-symmetry-preservation}, we obtain the desired result.
\end{proof}

Combining Proposition \ref{prop-ppf-operator} and Theorem \ref{thm-symmetry-preservation} leads to the symmetry of 
$\PPP_n(t)$, which is not obvious from the definition. 

\begin{cor}
For $n\geq 1$, the polynomial $\PPP_n(t)$ is symmetric with center of symmetry $n/2$.
\end{cor}

\section{Further results and conjectures}\label{sec 6}

Let us go back to Lascoux's open problem, which asks for a satisfactory expression of ${\mathcal{E}}_{\lambda,\sigma}(t)$.
As mentioned in the Introduction, Lascoux \cite{Lascoux-2016}  provided a combinatorial interpretation of  ${\mathcal{E}}_{\lambda,\sigma}(t)$
for $\lambda=\delta_n=(n-1,n-2,\ldots,1,0)$ and $\sigma=[n,n-1,\ldots,1]$ in terms of the classical Eulerian polynomials. 
The classical Eulerian polynomial $\mathcal{A}_n(t)$ is just $\mathcal{A}_{[n]}(t)$ defined by \eqref{eq-gf-multiset}.
Lascoux obtained the following result. 

\begin{prop}[{\cite[Corollary 6.3]{Lascoux-2016}}]
Given a Young subgroup $\mathfrak{S}_{n_1}\times \mathfrak{S}_{n_2}\times\cdots\times \mathfrak{S}_{n_r}$, if we let $\omega$ denote its element of maximal length $\ell=\binom{n_1}{2}+\binom{n_2}{2}+\cdots+\binom{n_r}{2}$ and let $n=n_1+n_2+\cdots+n_r$, then ${\mathcal{E}}_{\delta_n,\omega}(t)=\mathcal{A}_{\ell}(t)$. 
\end{prop}

Our result (Theorem \ref{main-thm-com-in}) gives an answer to Lascoux's problem for $\lambda=(n-1,n-2,\ldots,1,0)$ and $\sigma=[n, 1,2,\ldots,n-1]$ or $\sigma=[2,3,\ldots,n,1]$.
This result was further generalized to any strict partition $\lambda$ by Gao and Liu \cite{Gao-Liu-2026+}, who gave a combinatorial interpretation of ${\mathcal{E}}_{\lambda,[2,3,\ldots,n,1]}(t)$ and ${\mathcal{E}}_{\lambda,[n, 1,2,\ldots,n-1]}(t)$ in terms of $\mathbf{u}$-parking functions. Given a non-decreasing sequence $\mathbf{u}=(u_1,u_2,\ldots,u_n)$ of positive integers, recall that a $\mathbf{u}$-parking function of length $n$ is a sequence 
$\mathbf{a}=(a_1,a_2,\ldots,a_n)$ of positive integers such that its increasing rearrangement $b_1\leq b_2\leq \cdots\leq b_n$  satisfies $b_{i}\leq u_i$. Clearly, the classical parking functions correspond to $\mathbf{u}=(1,2,\ldots,n)$. 

We would like to point out a connection to Avila, Ferroni and Morales's recent work \cite{Avila-Ferroni-Morales-2026a} on the $h^*$-polynomials of Pitman–Stanley polytopes. 
For any lattice polytope $\mathcal{P}\subset\mathbb{R}^n$, the Ehrhart polynomial $i(\mathcal{P},\,j)$ counts the number of lattice points in its $j$-th dilation.
If $\mathcal{P}$ is of dimension $d$, then there exists a polynomial $h^*_{\mathcal{P}}(t)$ of degree at most d and satisfying
\begin{align}\label{eq-hstarpol}
\sum_{j\geq 0} i(\mathcal{P},\,j) t^j=\frac{h^*_{\mathcal{P}}(t)}{(1-t)^{d+1}}.
\end{align}
The polynomial $h^*_{\mathcal{P}}(t)$ is usually called the Ehrhart $h^*$-polynomial of $\mathcal{P}$. 
As noted in \cite[P. 4]{Avila-Ferroni-Morales-2026b}, a combinatorial interpretation for the $h^*$-polynomial of an arbitrary Pitman–Stanley
polytope was provided by Avila, Ferroni and Morales \cite{Avila-Ferroni-Morales-2026a}. By Remark \ref{rem-PS-operator} we know that ${\mathcal{E}}_{\lambda,[2,3,\ldots,n,1]}(t)$
appears as the Ehrhart $h^*$-polynomial $h^*_{\mathcal{P}}(t)$ for some Pitman–Stanley polytope $\mathcal{P}$. This means that Avila, Ferroni and Morales's work \cite{Avila-Ferroni-Morales-2026a} also 
provides a partial answer to Problem \ref{prob-Lascoux}. Avila, Ferroni and Morales \cite{Avila-Ferroni-Morales-2026b} also proved the following result. 

\begin{thm}[{\cite[Theorem 1.2]{Avila-Ferroni-Morales-2026b}}]
For any Pitman–Stanley polytope $\mathcal{P}$ the Ehrhart $h^*$-polynomial $h^*_{\mathcal{P}}(t)$ has only real zeros. 
\end{thm}

This theorem implies that the polynomial ${\mathcal{E}}_{\lambda,[2,3,\ldots,n,1]}(t)$ has only real zeros for any strict partition $\lambda$. 
Note that the real-rootedness of $\PP_n(t)={\mathcal{E}}_{\delta_n,[2,3,\ldots,n,1]}(t)$ is also known to Athanasiadis, Douvropoulos and Kalampogia-Evangelinou \cite[Theorem 3 (a)]{Athanasiadis-Douvropoulos-Kalampogia-Evangelinou-2024}.
It is well known that the Eulerian polynomial $\mathcal{A}_n(t)$ has only real zeros. 
These results in turn prompted the following conjecture.

\begin{conj}
For any partition $\lambda$ and any permutation $\sigma$ the polynomial ${\mathcal{E}}_{\lambda,\sigma}(t)$ has only real zeros. 
\end{conj}

Alexandersson (private communication) has pointed out that the above conjecture was independently stated in terms of $h^*$-polynomials in \cite{Alexandersson-Oguz-2025}. 
A useful tool to prove real-rootedness is the generalized Sturm sequence. A sequence $\{f_m(t)\}_{m\geq0}$ of real-rooted polynomials with nonnegative coefficients is called a generalized Sturm sequence if $f_m(t)\preceq f_{m+1}(t)$ for every $m\geq0$. 
Alexandersson and O\v{g}uz \cite{Alexandersson-Oguz-2025} also proposed the following conjecture.

\begin{conj}[\cite{Alexandersson-Oguz-2025}]\label{conj-AO}
If $\sigma^{(0)}<\sigma^{(1)}<\cdots<\sigma^{(\binom{n}{2})}$ is a saturated chain in the Bruhat order of the symmetric group $\mathfrak{S}_n$, then $\{\E_{\delta_n, \sigma^{(i)}}\}_{0\le i\le \binom{n}{2}}$ is a generalized Sturm sequence.
\end{conj}

Very little progress has been made towards the resolution of Conjecture \ref{conj-AO} so far. Note that $s_1s_2\cdots s_{n-1}=[2,3,\ldots,n,1]$ and $\PP_{n}(t)=\E_{\delta_m, [2,3,\ldots,n,1]}$ for any $m\geq n$. 
In support of Alexandersson and O\v{g}uz's conjecture, we prove the following result. 
 
\begin{thm}\label{thm_parking-sturm}
The sequence $\{\PP_{n}(t)\}_{n\geq 1}$ is a generalized Sturm sequence.
\end{thm}

We also have an analogous result for prime parking functions. 

\begin{thm}\label{thm_prime-sturm}
The sequence $\{\PPP_{n}(t)\}_{n\geq 1}$ is a generalized Sturm sequence.
\end{thm}

Before proving Theorems \ref{thm_parking-sturm} and \ref{thm_prime-sturm}, let us first recall some related results. 
Let $\mathcal T:\mathbb{R}[t]\rightarrow \mathbb{R}[t]$ be the linear operator defined by
\begin{align*}
\mathcal T\left(\binom{t}{i}\right)=t^i,
\end{align*}
for all $i\in \mathbb{N}$. 
As shown in \cite{Branden-2006},  for any $f\in\mathbb R[t]$ the polynomials $\mathcal{T}(f)(t)$ and $\mathcal{W}(f)(t)$ are related by
\begin{align}\label{relation_TW}
\mathcal T(f)(t)=(1+t)^{\deg(f(t))}\mathcal{W}(f)\left(\frac{t}{1+t}\right),
\end{align}
and for any $\alpha\in\mathbb{R}$ the following recursion holds
\begin{align}\label{eq_T_recursion}
\mathcal{T}\bigl((t-\alpha)f\bigr)(t) = (t-\alpha)\mathcal{T}(f)(t) + t(1+t)\mathcal{T}(f)'(t).
\end{align}
It is easy to see that $\mathcal{T}(f)(t)$ is $[-1,0]$-rooted if and only if $\mathcal{W}(f)(t)$ is $(-\infty,0]$-rooted.
Br\"and\'en \cite{Branden-2006} established a connection between the real-rootedness of $\mathcal{T}(f)(t)$ and that of
$\mathcal{T}\bigl((t-\alpha)f\bigr)(t)$, which plays an important role for proving Theorems \ref{thm_parking-sturm} and \ref{thm_prime-sturm}.

\begin{lem}[{\cite[Lemma 4.4]{Branden-2006}}]\label{lem_Branden}
Suppose that $f$ is a polynomial such that $\mathcal T(f)$ is $[-1,0]$-rooted. 
Then for any $\alpha\in[-1,0]$ the polynomial $\mathcal T((t-\alpha)f)$ is $[-1,0]$-rooted and moreover $\mathcal T(f)\preceq\mathcal T((t-\alpha)f)$. 
If $\mathcal{T}(f)$ in addition only has real and simple zeros, then so does $\mathcal{T}((t-\alpha)f)$.
\end{lem}

We also need a result of Fisk \cite{Fisk-2006}.

\begin{lem}[{\cite[Remark 1.21]{Fisk-2006}}]\label{lem_Fisk}
Suppose that $q(t)$ is a polynomial of degree $n$ with positive leading coefficient and simple real zeros $\{\theta_1,\ldots,\theta_n\}$. 
If $f(t)$ is a polynomial of degree $n-1$ with positive leading coefficient and we write
\begin{align*}
f(t)=\sum_{j=1}^n\frac{\alpha_j q(t)}{t-\theta_j},
\end{align*}
then $f(t)\preceq q(t)$ if and only if $\alpha_j\geq0$ for all $1\leq j\leq n$.
\end{lem}

By \eqref{eq-ppn-gnmv-def} and \eqref{eq-ppn-gnmv-def-prime}, Theorems \ref{thm_parking-sturm} and \ref{thm_prime-sturm} can be deduced from the following more general result. 

\begin{thm}\label{thm_gc}
Fix a nonnegative integer $c$ and let $\mathcal{G}_n(t)=g^{(0,1,\ldots,n-c)}_{n,n-c}(t)$ as defined by \eqref{eq:defi-fnk(x)} for any $n\geq c+1$. 
Then $\{\mathcal{G}_n(t)\}_{n\ge (c+1)}$ is a generalized Sturm sequence.
\end{thm}

\begin{proof}
As in the proof of Theorem \ref{thm-main-realzeros}, for each $0\le i\le n-c$ let $q_i(t):=nt+n-i$. Set $\mathcal{Q}_n(t)=\prod_{i=0}^{n-c}q_i(t)$ and hence
$\mathcal{G}_n(t)=\mathcal W\left(\mathcal{Q}_n\right)(t)$.
Put $d=n-c+1$. Let $\theta_1<\cdots<\theta_{d}$ and $\eta_1<\cdots<\eta_{d+1}$ be the zeros of  $\mathcal{Q}_n(t)$ and $\mathcal{Q}_{n+1}(t)$, respectively. 
It is easy to verify that 
\begin{align*}
-1=\eta_1=\theta_1<\eta_2<\theta_2<\cdots<\eta_{d}<\theta_{d}\leq\eta_{d+1}\le 0.
\end{align*}
Hence, $\mathcal{Q}_n(t)\preceq \mathcal{Q}_{n+1}(t)$.
By Lemma~\ref{lem_Fisk}, there exit nonnegative numbers $\alpha_1,\alpha_2,\ldots,\alpha_{d+1}$ such that
\begin{align}\label{eq_pc}
\mathcal{Q}_n(t)=\sum_{j=1}^{d+1}\alpha_j\frac{\mathcal{Q}_{n+1}(t)}{t-\eta_j}.
\end{align}
By Lemma \ref{lem_Branden}, for each $1\leq j\leq d+1$ we have 
$\mathcal{T}\left(\frac{\mathcal{Q}_{n+1}(t)}{t-\eta_j}\right)\preceq \mathcal{T}(\mathcal{Q}_{n+1})(t)$, and hence $\mathcal{T}(\mathcal{Q}_{n})(t)\preceq \mathcal{T}(\mathcal{Q}_{n+1})(t)$.
Together with \eqref{relation_TW} this implies that $\mathcal{W}(\mathcal{Q}_{n})(t)\preceq \mathcal{W}(\mathcal{Q}_{n+1})(t)$, or equivalently $\mathcal{G}_n(t)\preceq \mathcal{G}_{n+1}(t)$, as desired.
\end{proof}

\noindent \textbf{Acknowledgements.} 
We are grateful to Per Alexandersson for some helpful discussions on key polynomials. 
Alice Gao is supported by the National Science Foundation of China (No.11801447) and the National Science Foundation for Post-doctoral Scientists
of China (No.2020M683544). 
Xin-Bei Liu is supported by the Fundamental Research Funds for the Central Universities.
Arthur Yang is supported by the National Science Foundation of China (No.\ 12325111). 


\end{document}